\documentclass[11pt,a4paper,reqno]{amsart}
\fontsize{6.5pt}{6.5pt}\selectfont 
\makeatletter
\def\@evenhead{{\fontsize{6.5pt}{6.5pt}\selectfont \hfil \leftmark\hfil\thepage}}
\def\@oddhead{{\fontsize{6.5pt}{6.5pt}\selectfont\hfil\rightmark \hfil\thepage}}
\makeatother

\usepackage{amsfonts}
\usepackage{booktabs}
\usepackage{amsmath,amssymb,amsthm,amsxtra}
\usepackage{mathtools}
\usepackage{float}
\usepackage{amssymb}
\usepackage{mathabx}
\usepackage{bbm}
\usepackage[colorlinks, linkcolor=blue,anchorcolor=Periwinkle,
citecolor=Red,urlcolor=Emerald]{hyperref}
\usepackage[usenames,dvipsnames]{xcolor}
\usepackage{enumitem}
\usepackage{geometry,array} 
\usepackage{graphicx}
\usepackage{subfigure}
\usepackage{bookmark}
\usepackage{tikz}\usetikzlibrary{matrix}
\usepackage{url}
\usepackage{dsfont}
\usepackage[colorinlistoftodos]{todonotes}
\usepackage{tablists} \restorelistitem
\usepackage{adjustbox}
\usepackage{capt-of}

\usetikzlibrary{decorations.markings}
\tikzset{->-/.style={decoration={  markings,  mark=at position #1 with
    {\arrow{>}}},postaction={decorate}}}
\tikzset{-<-/.style={decoration={  markings,  mark=at position #1 with
    {\arrow{<}}},postaction={decorate}}}
\usepackage{extarrows}
\usepackage[all]{xy}
\usepackage{setspace}
\usepackage{thmtools}
\usepackage{thm-restate}
\usepackage{hyperref}
\usepackage{cleveref}
\usepackage{multirow}

\newcommand{\sgn}{\operatorname{sgn}}

\newcommand{\diag}{\operatorname{diag}}
\newcommand{\mfb}{\mathbf{b}}

\newcommand{\mfx}{\mathbf{x}}

\newcommand{\mcA}{\mathcal{A}}

\newcommand{\mcF}{\mathcal{F}}

\newcommand{\mcT}{\mathcal{T}}
\newcommand{\mcU}{\mathcal{U}}

\newcommand{\mcZ}{\mathcal{Z}}
\newcommand{\mbN}{\mathbb{N}}

\newcommand{\mbQ}{\mathbb{Q}}
\newcommand{\mbR}{\mathbb{R}}

\newcommand{\mbT}{\mathbb{T}}

\newcommand{\mbZ}{\mathbb{Z}}
\theoremstyle{plain}
\newtheorem{theorem}{Theorem}[section]

\newtheorem{lemma}[theorem]{Lemma}

\newtheorem{proposition}[theorem]{Proposition}
\newtheorem{conjecture}[theorem]{Conjecture}
\theoremstyle{definition}
\newtheorem{definition}[theorem]{Definition}

\newtheorem{example}[theorem]{Example}

\newtheorem{remark}[theorem]{Remark}

\newtheorem{question}[theorem]{Question}

\numberwithin{equation}{section}
\newtheorem{definition-proposition}[theorem]{Definition-Proposition}

\begin{document}

\title{Mutation-preserving generalized cluster algebras and Laurent mutation invariants}

\date{\today}

\author{Zhichao Chen}
\address{School of Mathematical Sciences\\ University of Science and Technology of China \\ Hefei, Anhui 230026, P. R. China.}
\email{czc98@mail.ustc.edu.cn}
\author{Yimin Huang}
\address{School of Mathematical Sciences\\ Fudan University\\ Shanghai 200082, P. R. China.}
\email{1103165219@qq.com} 

%=====================================
\begin{abstract}
	We introduce mutation-preserving generalized cluster algebras, for which the generalized cluster mutation in each direction is independent of the seed in the mutation equivalence class. We classify all irreducible generalized cluster algebras with this property. Then, a Markov-type Diophantine equation
	 $x^2+y^2+z^2+2yz=kxyz$ is studied, which has a structure of the mutation-preserving generalized cluster algebra.
	We prove that positive integer solutions exist if and only if $1\leq k\leq 5$ and determine all their orbits under the associated generalized cluster mutation groups. In particular, multiple orbits occur for each $k=1,3$, whereas the solutions form a single orbit for each $ k=2,4,5$. We then classify all generalized Markov Laurent mutation invariants. As an application, a conjecture proposed by Chen-Li is proved, showing that every Laurent mutation invariant of irreducible sign-equivalent cluster algebras is essentially a polynomial in the corresponding basic invariant.
\end{abstract}
\maketitle
%=======================================
\vspace{-2.25\baselineskip}
\tableofcontents
%=======================================
\newpage
\section{Introduction}\label{sec: introduction}

Cluster algebras were introduced by Fomin and Zelevinsky \cite{FZ02,FZ03} in connection with total positivity in Lie groups and canonical bases in quantum groups. Since then, more applications were found in representation theory \cite{BMRRT06,BIRS09, DWZ08, DWZ10, HL10,KY11}, higher Teichm\"uller theory \cite{FG09, GS26}, Poisson geometry \cite{GSV03, GSV12}, integrable systems \cite{KNS11}, commutative algebra \cite{Mul13,Mul14}, combinatorics \cite{FR05,NZ12,LLZ14,RS18,GHKK18,FG19} and number theory \cite{Pro20, PZ12, Lam16, Hua22, LLRS23, BL24, Kau24, CL24, CL25, BL25, CJ25, Mus25, GKW25, BH26}. 

Generalized cluster algebras, introduced by Chekhov-Shapiro \cite{CS14} and Nakanishi \cite{Nak15}, extend classical cluster algebras by replacing binomial exchange relations with higher degree exchange polynomials. They encode the data of generalized seeds $(B;R;\mcZ)$ and generalized mutations, where $B$ is a generalized exchange matrix, $R$ is an exchange degree matrix, and $\mcZ$ is a collection of exchange polynomials. Despite this greater flexibility, they share two fundamental properties with classical cluster algebras: the Laurent phenomenon and positivity. Namely, every generalized cluster variable is a Laurent polynomial in the variables of any initial generalized cluster with non-negative coefficients \cite{CS14, Nak15, BLM25}. Thus, generalized cluster algebras preserve essential features of classical cluster theory while providing a natural framework for mutation phenomenon beyond the classical setting; see also \cite{NR16,Nak23}.

In a generalized cluster pattern, mutation at each seed induces a rational transformation of the ambient rational function field. In general, this transformation may depend on the seed, since both the exchange matrix and the exchange polynomial can change under generalized mutation. We call a generalized cluster algebra \emph{mutation-preserving} if, in each fixed direction, the induced rational transformation is the same at every seed in the mutation equivalence class (\Cref{def: mutation-preserving}). In this case, these seed-independent transformations  induce a group action on the rational function field. Rational functions invariant under this action, called \emph{generalized mutation invariants}, then provide a natural bridge between generalized cluster algebras and Diophantine equations; see \Cref{mutation invariant} and \Cref{equivalent}.

Generalized Markov equations (\Cref{equ: gme}) studied in \cite{GM23} indeed admit mutation-preserving generalized cluster algebraic structures. For these equations, all positive integer solutions are generated from $(1,1,1)$ by generalized cluster mutations and hence form a single orbit of the generalized cluster mutation group $\widehat{\Gamma}$; see \Cref{def: gmp}.  
\begin{question}\label{question}
This leads to the following three natural questions.
\begin{enumerate}
	\item What is the classification of mutation-preserving generalized cluster algebras?
	\item Is there some Diophantine equation whose positive integer solutions can be generated from an initial solution distinct from $(1,1,1)$ under the action of the generalized cluster mutation group $\widehat{\Gamma}$? Moreover, does there exist a Diophantine equation whose positive integer solutions can be decomposed into multiple $\widehat{\Gamma}$-orbits of initial solutions?
	\item Classify the generalized Laurent mutation invariants associated with the generalized Markov equations. In particular, solve the conjectures of Chen-Li in \cite{CL25, CL24}.
\end{enumerate}
\end{question}

Our first objective is to classify all irreducible mutation-preserving generalized cluster algebras. The problem is closely related to sign-equivalent exchange matrices, whose mutation equivalence classes contain only two elements (\Cref{def sign}). The irreducible sign-equivalent exchange matrices were classified in \cite{CL25}. Combining this result with exchange degree matrices and  exchange polynomials, we obtain the following complete classification.

\begin{theorem}[= \Cref{thm: classification}]
	All the irreducible mutation-preserving generalized cluster algebras $\mcA_{\mathbf{GCA}}=\mcA(B;R;\mcZ)$ are as follows:
	\begin{enumerate}
		\item $\mcA(B;R;\mcZ)$ is of rank $2$ with arbitrary coefficients for $\mcZ$.
		\item $\mcA(B;R;\mcZ)$ is of rank $3$ with the pair $(B;R)$ given in \Cref{tab: rank-three-classification} except (II.2) and arbitrary coefficients for $\mcZ$.
		\item $\mcA(B;R;\mcZ)$ is of rank $3$ with the pair $(B;R)$ given in (II.2) and the coefficients for $\mcZ$  satisfy the reciprocity condition. 
	\end{enumerate}
\end{theorem}

The interaction between cluster mutations and Diophantine equations goes back to the classical Markov equation \cite{Mar80}
\begin{align*}
x^2+y^2+z^2=3xyz.
\end{align*}
Its positive integer solutions are generated from $(1,1,1)$ by Vieta jumping, which coincide with the mutations of the once-punctured torus cluster algebra; see \cite{Pro20,BBH11,PZ12,LLRS23,Hua22}. Thus, the cluster structure organizes the positive integer solutions into a mutation orbit. Lampe \cite{Lam16} found a similar cluster structure for another Markov-type equation, and subsequent work further developed generalized Markov equations and mutation invariants \cite{GM23, BL24, Kau24, CL24,CL25, CJ25, BL25, Mus25, GKW25}. These developments suggest that generalized mutation invariants provide a natural framework not only for generating positive integer solutions, but also for understanding their orbit decompositions. This motivates us to seek, in the setting of mutation-preserving generalized cluster algebras, a Markov-type equation exhibiting the more intricate orbit structure proposed in \Cref{question}.

Our second objective is to answer this question by considering the generalized cluster algebra of type (I.2) in \Cref{tab: rank-three-classification}.
Its generalized mutation invariant gives rise to the Markov-type Diophantine equation
\begin{align}\label{eq: main introduction}
x^2+y^2+z^2+2yz=kxyz,
\end{align}
where $k\in \mbN$. Using generalized cluster mutations together with a descent on the maximal component of solutions, we determine precisely when \Cref{eq: main introduction} has positive integer solutions and classify all their orbits of solutions; see also \Cref{table: markov-type equation}.

\begin{theorem}[= \Cref{thm: main} and \Cref{thm: 12345}]
Equation \eqref{eq: main introduction} has a positive integer solution if and only if $1\leq k\leq5$. More precisely, \begin{enumerate}
	\item For $k=1$, its positive integer solutions form four generalized cluster mutation orbits: $\widehat{\Gamma}[(9,3,6)],\ \widehat{\Gamma}[(9,6,3)],\
\widehat{\Gamma}[(8,4,4)],\
\widehat{\Gamma}[(5,5,5)]$;
\item For $k=3$, its positive integer solutions form two generalized cluster mutation orbits: $\widehat{\Gamma}[(3,2,1)],\ \widehat{\Gamma}[(3,1,2)]$;
	\item For $k=2,4,5$,  the positive integer solution set is the single orbit  represented respectively by $
\widehat{\Gamma}[(4,2,2)],\ \widehat{\Gamma}[(2,1,1)],\ \widehat{\Gamma}[(1,1,1)]$.
\end{enumerate}
\end{theorem}

As shown in \Cref{prop: scaling correspondence}, the four orbits of solutions for $k=1$ are related by precise scaling correspondences to the orbits of solutions for $k=2,3,4,5$. This result may be viewed as an analogue of Hurwitz's result for the classical Markov equation \cite{Hur07}; see also \cite[Proposition 2.2]{Aig13}.

There is also a close connection to the classical cluster algebra with the exchange matrix given by \eqref{4-matrix}.
The map $S(x,y,z)=(x^2,y,z)$ connects the classical cluster mutations with the generalized ones. Their compatibility (\Cref{contract}) allows us to transfer the orbit classification to a second family of Diophantine equations with a classical cluster algebraic structure (\Cref{tab: classical}).

\begin{theorem}[= \Cref{thm: classical markov type}]
The equation
\begin{align*}
x^4+y^2+z^2+2yz=kx^2yz
\end{align*}
has positive integer solutions if and only if $k=1,2,5$. For $k=1$, there are exactly two orbits of solutions: $\Gamma[(3,6,3)]$ and $\Gamma[(3,3,6)]$. For $k=2$ and $k=5$, there is only one orbit of solutions: $\Gamma[(2,2,2)]$ and $\Gamma[(1,1,1)]$, respectively.
\end{theorem}

Our third objective is to classify the generalized Markov Laurent
mutation invariants associated with the mutation-preserving generalized
cluster algebra of type (I.4) in
\Cref{tab: rank-three-classification}. Mutation invariants encode
algebraic constraints shared by every seed in a mutation orbit, so their
classification provides a global description of the corresponding
generalized cluster mutation group action. Related invariant structures
arising from quiver mutation have also been studied through cyclically
ordered quivers and long mutation cycles \cite{FN26,FN25}.

For $k_1,k_2,k_3\in\mbN$, the basic generalized Markov Laurent mutation
invariant of type (I.4) is
\begin{align*}
\widehat\mcT
=\frac{x^2+y^2+z^2+k_1yz+k_2xz+k_3xy}{xyz}.
\end{align*}
This invariant arises from the family of generalized Markov equations
studied by Gyoda and Matsushita in \cite{GM23}.
By combining techniques from field extension theory, commutative
algebra, and generalized cluster mutation theory, we show that every
generalized Markov Laurent mutation invariant is a polynomial in
$\widehat\mcT$. As an application, we  prove a
conjecture of Chen-Li \cite[Conjecture~6.9]{CL25} concerning the
Laurent mutation invariants of rank $3$ cluster algebras with
irreducible sign-equivalent exchange matrices in \eqref{thm: sign-equivalent}; see also \Cref{conj: rank 3}.

\begin{theorem}[= \Cref{thm:I4-generalized-invariant-ring} and
\Cref{thm:classical-from-I4}]
Let $\mcF=\mathbb Q(x,y,z)$,
$\mathcal L=\mathbb Q[x^{\pm1},y^{\pm1},z^{\pm1}]$, and
$\widehat\Gamma=\langle\widehat\mu_1,\widehat\mu_2,
\widehat\mu_3\rangle$ be the generalized cluster mutation group of
type (I.4) in
\Cref{tab: rank-three-classification}. Then,
\begin{align*}
\mcF^{\widehat\Gamma}=\mathbb Q(\widehat\mcT),\
\mcF^{\widehat\Gamma}\cap\mathcal L
=\mathbb Q[\widehat\mcT],
\end{align*} where $\mcF^{\widehat\Gamma}=\{f\in\mcF\mid f(\widehat\mu_i(x,y,z))=f,\ i=1,2,3\}$.
In particular, \Cref{conj: rank 3} holds.
\end{theorem}

These results lead to further natural problems. We consider broader
families of generalized Markov-type equations, seeking criteria for the
existence of positive integer solutions and a description of their
mutation orbit decompositions (\Cref{question: open}). We also formulate
ordered uniqueness conjectures for positive integer solutions and record a
counterexample showing the necessity of certain hypotheses
(\Cref{conj: uniqueness-CH}).

The paper is organized as follows. In \Cref{S2}, we review the necessary background on generalized cluster algebras. Furthermore, we introduce generalized mutation invariants and mutation-preserving generalized cluster algebras (\Cref{mutation invariant} and \Cref{def: mutation-preserving}). In \Cref{sec: classification}, we classify all irreducible mutation-preserving generalized cluster algebras (\Cref{thm: classification}). In \Cref{sec: diophantine}, we study the Markov-type equation \eqref{eq: main introduction}, establish the criterion for the existence of positive integer solutions, and classify their mutation orbits (\Cref{prop: k larger than 5}, \Cref{thm: main} and \Cref{thm: 12345}). We also obtain the corresponding result for a Diophantine equation with a cluster algebraic structure (\Cref{thm: classical markov type}). In \Cref{sec:laurent-classification}, we classify generalized Markov Laurent mutation invariants (\Cref{thm:I4-generalized-invariant-ring}). In \Cref{sec:proof-chen-li-conjecture}, we prove a conjecture proposed by Chen-Li (\Cref{thm:classical-from-I4}). In \Cref{sec: further-questions}, we present an open question, a uniqueness conjecture, and a counterexample. In \Cref{sec:orbit-appendix}, we illustrate the mutation orbits appearing in \Cref{thm: main}.
\vspace{-1.0\baselineskip}
\begingroup
\makeatletter
\let\@tocwrite\@gobbletwo
\section*{Acknowledgements}
\makeatother
\endgroup
The authors are grateful to Tomoki Nakanishi, Zhe Sun and Yu Ye for their valuable suggestions and comments. Z. Chen also sincerely thanks  Nagoya University for its 2-year hospitality. This work is supported by National Natural Science Foundation of China (Grant No. 124B2003) and China Scholarship Council (Grant No. 202406340022). 
\vspace{3mm}
\vspace{0.1\baselineskip}
\begingroup
\makeatletter
\let\@tocwrite\@gobbletwo
\section*{Conventions}
\makeatother
\endgroup

Throughout this paper, we use the following notations and conventions.
\begin{itemize}[leftmargin=2em]
	\item The sets of integers, non-negative integers, positive integers, rational numbers, and real numbers are denoted by $\mbZ$, $\mbN$, $\mbN_{+}$, $\mbQ$, and $\mbR$, respectively. We write $\mbQ_{+}$ for the set of positive rational numbers.
	
	\item Let $\operatorname{Mat}_{m\times n}(\mbR)$ be the set of all $m\times n$ matrices with entries in $\mbR$. We denote by $I_n$ and $O$ the $n\times n$ identity matrix and a zero matrix of an appropriate size, respectively. The diagonal matrix with diagonal entries $a_1,\dots,a_n$ is denoted by $\diag(a_1,\dots,a_n)$.
	
	\item For $a\in\mbR$, we set
	\[
	[a]_{+}=\max\{a,0\},\quad
	\sgn(a)=
	\begin{cases}
	1,&a>0,\\
	0,&a=0,\\
	-1,&a<0.
	\end{cases}
	\]
	
	\item The symmetric group on $\{1,\dots,n\}$ is denoted by $\mathfrak{S}_n$. 
	
	\item If a group $G$ acts on a set and $p$ is an element of that set, then $G[p]$ denotes the $G$-orbit of $p$. In particular, $\Gamma[p]$ and $\widehat{\Gamma}[p]$ denote the orbits of $p$ under the classical and generalized cluster mutation groups, respectively.
\end{itemize}

\section{Preliminaries}\label{S2}
In this section, we review some basic notions about generalized cluster algebras, permutation and irreducibility of matrices. Then, we introduce generalized mutation invariants and mutation-preserving generalized cluster algebras.

\subsection{Generalized cluster algebras}\

In this subsection, we first recall some definitions and properties about the \emph{generalized cluster algebras} (GCA, for short), based on \cite{Nak15, CS14}.
%\begin{definition}[\emph{Mutation data}]
%\end{definition}
\begin{definition}[\emph{Generalized seed}]\label{Def of GS} Let $\mcF$ be the field of rational functions in $n$ independent variables with coefficients in $\mbQ$. A (labeled) \emph{generalized seed} of rank $n$ is a triple $(\mfx,B,\mcZ)$, where
\begin{itemize}[leftmargin=2em]
\item $\mathbf{x}=(x_1, \dots, x_n)$ is an $n$-tuple of algebraically independent generators of $\mcF$.
\item $B=(b_{ij})_{n\times n}$ is a skew-symmetrizable matrix. 
\item $\mcZ=(Z_1,\dots,Z_n)$ is an $n$-tuple of polynomials over $\mbN$, where 
\begin{align*}Z_i(u)=z_{i,0}+z_{i,1}u+\cdots+z_{i,r_i}u^{r_i},\end{align*}
such that $z_{i,0}=z_{i,r_i}=1$, which is called the \emph{monic condition}.
\end{itemize} Here, we respectively call $\mfx$ a \emph{generalized cluster}, $x_i$ a \emph{generalized cluster variable}, $B$ a \emph{generalized exchange matrix}, $Z_i$ an \emph{exchange polynomial} and $r_i$ an \emph{exchange degree}. Moreover, let $R=\diag(r_1,\dots, r_n)$. Then, it is a diagonal matrix with positive integer entries, which is called an \emph{exchange degree matrix}.  
\end{definition}
Note that $BR$ is still a skew-symmetrizable matrix with the skew-symmetrizer $RD$.
\begin{definition}[\emph{Generalized mutation}]
Let $(\mfx,B,\mcZ)$ be a generalized seed and $k\in \{1,\dots,n\}$. We define another generalized seed $\mu_{k;\mathbf{GCA}}(\mfx,B,\mcZ)=(\mfx^{\prime},B^{\prime},\mcZ^{\prime})$, such that 
	\begin{itemize}[leftmargin=2em]
	\item The generalized cluster variables $(x_1^{\prime},\dots,x_n^{\prime})$ are given by 
	\begin{align}\ 
		x_{i}^{\prime}=\left\{
		\begin{array}{ll}
			x_{k}^{-1}\left(\mathop{\prod}\limits_{j=1}^{n} x_j^{[-b_{jk}]_+}\right)^{r_k}Z_k\left(\mathop{\prod}\limits_{j=1}^{n} x_j^{b_{jk}}\right), &   \text{if}\ i=k, \\
			x_{i}, &  \text{if}\ i \neq k, 
		\end{array} \right. \label{eq: generalized cluster variables}
	\end{align}

	\item The entries of $B^\prime=(b^\prime_{ij})_{n\times n}$ are given by \begin{align} \label{generalized matrix mutation}
		b_{ij}^{\prime}=\left\{
		\begin{array}{ll}
			-b_{ij}, &  \text{if}\ i=k \;\;\mbox{or}\;\; j=k, \\
			b_{ij}+r_k([b_{ik}]_{+}b_{kj}+b_{ik}[-b_{kj}]_{+}), &  \text{if}\ i\neq k \;\;
			\mbox{and}\; j\neq k. 
		\end{array} \right. 
	\end{align}

	\item The exchange polynomials $\mcZ^{\prime}=(Z^{\prime}_1, \dots, Z^{\prime}_n)$ are given by $Z_i^{\prime}=Z_i$ for $i\neq k$ and for $Z_k^{\prime}$, its coefficients satisfy $z^{\prime}_{k,s}=z_{k,r_k-s}$.
	\end{itemize}
\end{definition} Here, $\mu_{k;\mathbf{GCA}}$ is called a \emph{generalized cluster mutation in the direction $k$}.
Note that the mutation of exchange polynomials $\mcZ^{\prime}=(Z^{\prime}_1, \dots, Z^{\prime}_n)$ can also be given by
		\begin{align*}
		Z_{i}^{\prime}(u)=\left\{
		\begin{array}{ll}
			u^{r_k}Z_k(u^{-1}), &  \text{if}\ i=k , \\
			Z_i(u), &  \text{if}\ i\neq k . 
		\end{array} \right. 
	\end{align*} Sometimes, we might assume that $z_{i,s}=z_{i,r_i-s}$ for any $i$, which is called the \emph{reciprocity condition} \cite{CS14, Nak15}. In particular, when $r_i=1$ or $r_i=2$, such condition automatically holds for arbitrary coefficients. 
\begin{remark}\label{no coe GCA}	
Here, for our purposes, we only consider the generalized seeds and generalized mutations without coefficients in some semifield. We can refer to \cite{CS14, Nak15, NR16} for the general case with coefficients. 
\end{remark}
It can be checked directly that $(\mfx^{\prime},B^{\prime},\mcZ^{\prime})$ is still a generalized seed and $\mu_{k;\mathbf{GCA}}$ is involutive, that is $\mu_{k;\mathbf{GCA}}(\mfx^{\prime},B^{\prime},\mcZ^{\prime})=(\mfx,B,\mcZ)$, see \cite{Nak15, NR16, Nak23}. Then, similar to the classical cluster pattern $\mathbf{\Sigma}$, we can define the \emph{generalized cluster pattern} $\mathbf{\Sigma}_{\mathbf{GCA}}=\{(\mfx_t,B_t,\mcZ_t)|\ t\in \mbT_n\}$ to be a collection of generalized seeds which are labeled by the vertices of the $n$-regular tree $\mbT_{n}$ and connected by a single generalized mutation. We denote them by $\mathbf{x}_{t}=(x_{1;t},\dots,x_{n;t})$, $B_{t}=(b_{ij;t})_{n\times n}$ and $\mcZ_t=(Z_{1;t},\dots,Z_{n;t})$. We can fix an arbitrary initial vertex $t_0\in \mbT_n$. 

Two seeds are said to be \emph{mutation-equivalent} if they can be obtained from each other by a finite sequence of generalized mutations. The corresponding equivalence class is called the \emph{mutation equivalence class}.
\begin{definition}[\emph{Generalized cluster algebra}]\label{def of GCA} For a generalized cluster pattern $\mathbf{\Sigma}_{\mathbf{GCA}}$, the  \emph{generalized cluster algebra} $\mcA(\mathbf{\Sigma_{\mathbf{GCA}}})=\mcA(B;R;\mcZ)$ is the $\mbQ$-subalgebra of $\mcF$ generated by all the generalized cluster variables $\{x_{i;t}|\ i=1,\dots,n;t\in\mbT_{n}\}$. 
\end{definition}
\begin{example}[\emph{Type $B_2$}]
	Let the initial cluster be $\mfx=(x_1,x_2)$ and the initial triple $(B,\mcZ,R)$ be as follows:
\begin{align*}
	B=\begin{pmatrix}0 & -1 \\ 1 & 0 \end{pmatrix}, \left\{
		\begin{array}{ll}
			Z_1(u)=1+u+u^2 \\
			Z_2(u)=1+u 
		\end{array}, \right. R=\begin{pmatrix}2 & 0 \\ 0 & 1\end{pmatrix}.
\end{align*} Then, all the exchange matrices are the same up to a sign and all the exchange polynomials are invariant under the mutations. By a direct calculation, all the $6$ distinct generalized clusters are as follows:
\begin{align*}
\begin{array}{ll}
	\left\{\begin{array}{ll}
			x_{1;0}=x_1 \\
			x_{2;0}=x_2 
		\end{array}, \right. \left\{\begin{array}{ll}
			x_{1;1}=\frac{1+x_2+x_2^2}{x_1} \\
			x_{2;1}=x_2 
		\end{array}, \right. \left\{\begin{array}{ll}
			x_{1;2}=\frac{1+x_2+x_2^2}{x_1} \\
			x_{2;2}=\frac{1+x_2+x_2^2+x_1}{x_1x_2} 
		\end{array}, \right. \\ 
		\left\{\begin{array}{ll}
			x_{1;3}=\frac{1+2x_1+x_1^2+x_1x_2+x_2+x_2^2}{x_1x_2^2} \\
			x_{2;3}=\frac{1+x_2+x_2^2+x_1}{x_1x_2} 
		\end{array}, \right. \left\{\begin{array}{ll}
			x_{1;4}=\frac{1+2x_1+x_1^2+x_1x_2+x_2+x_2^2}{x_1x_2^2} \\
			x_{2;4}=\frac{1+x_1}{x_2} 
		\end{array}, \right. \left\{\begin{array}{ll}
			x_{1;5}=x_1 \\
			x_{2;5}=\frac{1+x_1}{x_2} 
		\end{array}. \right.
\end{array} 
\end{align*} Then, $\mcA(B;R;\mcZ)=\mbQ[x_{i;j}\ |\ i=1,2;\ j=0,1,2,3,4,5]\subseteq \mbQ(x_1,x_2)$.
\end{example} %We can refer to \cite{Nak23} for more examples of rank $2$ generalized cluster algebras, including other Dynkin types.
\begin{remark}
	In particular, when $R=I_n$, the generalized cluster pattern reduces to the classical cluster pattern defined in \cite{FZ07}, where we denote the mutation by $\mu_k$.
\end{remark}
\begin{lemma}\label{CA GCA mutation} Let $(\mfx,B,\mcZ)$ be a generalized seed and $k\in \{1,\dots,n\}$. Then, the equality holds 
	\begin{align} \label{eq: recover}
		\mu_{k;\mathbf{GCA}}(B)R=\mu_k(BR).
		\end{align}
\end{lemma}
%\begin{proof}
%	It is direct by the classical mutation rules \eqref{matrix mutation} and the generalized mutation rules \eqref{generalized matrix mutation}.
%\end{proof}
For brevity, we sometimes denote the generalized cluster algebra $\mcA(\mathbf{\Sigma_{\mathbf{GCA}}})$ by $\mcA_{\mathbf{GCA}}$ or $\mcA_{\mathbf{GCA}}(B)$, and its mutation $\mu_{k;\mathbf{GCA}}$ by $\widehat{\mu}_k$ without ambiguity.

\begin{definition}[\emph{Sign-equivalent matrix}]\label{def sign}
	An exchange matrix $B$ in a generalized seed is said to be  \emph{sign-equivalent} if its mutation equivalence class is $[B]=\{B,-B\}$. 
\end{definition} 
Note that all the rank $2$ exchange matrices and the exchange matrix associated with the Markov quiver are sign-equivalent.
\subsection{Permutation matrices and irreducible matrices}\

We recall some notions and properties about permutation matrices and irreducible matrices according to \cite[Chapter XIII]{Gan98}. 
\begin{definition}[\emph{Permutation matrix}]
	A permutation matrix $P$ is a square matrix that has exactly one entry of 1 in each row and each column with all other entries 0. In particular, the identity matrix is a special permutation matrix.
\end{definition} 
\begin{definition}[\emph{Reducible matrix}]
Let $B\in \text{Mat}_{n\times n}(\mbR)$.  It is said to be \emph{reducible} if there is a permutation matrix $P$, such that \begin{align*}P^{T} B P=\left(\begin{array}{ll}{B_1} & {B_2} \\ {O} & {B_3}\end{array}\right), \end{align*} where $B_1\in \text{Mat}_{k\times k}(\mbR)$ and $B_3 \in \text{Mat}_{(n-k)\times (n-k)}(\mbR)$ with $1\leq k\leq n-1$. Otherwise, it is said to be \emph{irreducible}. In particular, assume that $B$ is an exchange matrix. Since $B$ is sign-skew-symmetric, it is reducible if and only if $B_2=O$. Then, $P^{T} B P$ is  called the \emph{direct sum} of $B_1$ and $B_3$. 
\end{definition}
\begin{definition}[\emph{Irreducible generalized cluster algebra}]
	A generalized cluster algebra $\mcA_{\mathbf{GCA}}(B)$ is said to be \emph{irreducible} if the generalized exchange matrix $B$ is irreducible. Equivalently, all the generalized exchange matrices in the mutation equivalence class are irreducible.
\end{definition}
Note that if $B$ is the direct sum of $B_1$ and $B_2$, then there is a canonical algebraic isomorphism \begin{align*}\mcA_{\mathbf{GCA}}(B) \cong \mcA_{\mathbf{GCA}}(B_1) \times \mcA_{\mathbf{GCA}}(B_2).\end{align*}
For a permutation $\sigma\in \mathfrak{S}_n$, we define the left action of $\sigma$ on $B$ by $\sigma(B)=B^{\prime}=(b_{ij}^{\prime})$, such that $b^{\prime}_{ij}=b_{\sigma^{-1}(i)\sigma^{-1}(j)}.$ In fact, there is a permutation matrix corresponding to $\sigma$ \begin{align*}
P_\sigma=(p_{ij})_{n\times n},\ \text{where}\ p_{ij}=\delta_{i,\sigma^{-1}(j)},
\end{align*} such that 
$\sigma(B)=P^{T}_{\sigma}BP_{\sigma}.$ We then extend such action to the more general seed as follows.
\begin{definition}[\emph{Permutation action}]
	For a generalized seed $\Sigma=(\mfx,B,\mcZ)$ and a permutation $\sigma \in \mathfrak{S}_n$, we define the (left) action of $\sigma$ on $\Sigma$ by 
	\begin{align*}
		\sigma(\Sigma)=(\sigma(\mfx),\sigma(B),\sigma(\mcZ)),
	\end{align*} where $\sigma(\mfx)=\mfx^{\prime}$ is defined by $x^{\prime}_i=x_{\sigma^{-1}(i)}$ and $\sigma(\mcZ)=\mcZ^{\prime}$ is defined by $Z^{\prime}_i=Z_{\sigma^{-1}(i)}$. At the same time, $\sigma(R)=R^{\prime}=\diag(r^{\prime}_1,\dots, r^{\prime}_n)$ with $r^{\prime}_i=r_{\sigma^{-1}(i)}$. Then, the following proposition can be obtained directly by a simple calculation.
\end{definition}
\begin{proposition}[cf. \cite{Nak23}]
	For any $\sigma \in \mathfrak{S}_n$ and $k\in \{1,\dots,n\}$, the following compatible relation holds:
	\begin{align*}
		\widehat{\mu}_{\sigma(k)}(\sigma \Sigma)=\sigma (\widehat{\mu}_{k}(\Sigma)), 
	\end{align*} where $\Sigma$ is any generalized seed.
\end{proposition}
\subsection{Generalized mutation invariants of generalized cluster algebras}\

Given a generalized cluster algebra $\mcA_{\mathbf{GCA}}$ of rank $n$, we introduce the notion of generalized mutation invariants, which extends the definition of mutation invariants in  \cite{CL24}.
\begin{itemize}[leftmargin=2em]
	\item For the classical cluster algebras, 
the Laurent phenomenon was proved by \cite{FZ02} and the positivity theorem was obtained by \cite{GHKK18}. 
\item For the generalized cluster algebras, the Laurent phenomenon was proved by \cite{CS14, Nak15} and the positivity theorem was given by \cite{BLM25}. 
\end{itemize}
Hence, all the generalized clusters labeled by $t\in \mbT_n$ can be written as  
\begin{align*}\mfx_t=(c_{1;t}(x_1,\cdots,x_n),\cdots,c_{n;t}(x_1,\cdots,x_n)),\end{align*} where each $c_{i;t}(x_1,\cdots,x_n)$ is a Laurent polynomial with respect to the initial generalized cluster variables $x_i$ with non-negative integer coefficients.
\begin{definition}[\emph{Generalized mutation invariant}]\label{mutation invariant}
	 A non-constant and reduced rational function $\mcT(x_1,\dots,x_n)$ is called a \emph{generalized mutation invariant of $\mcA_{\mathbf{GCA}}$} if for any $t \in \mbT_{n}$,
	\begin{equation*}
		\mcT(x_{1},\cdots,x_{n})=\mcT(x_{1;t},\cdots,x_{n;t}).
	\end{equation*}
	In particular, if $\mcT(x_1,\dots,x_n) \in \mbQ[x_{1}^{\pm 1},\dots,x_{n}^{\pm 1}]$, it is called a \emph{generalized Laurent mutation invariant}. Moreover, for a classical cluster algebra, we may call it a \emph{Laurent mutation invariant} without ambiguity.
\end{definition} 
\begin{remark}
	For rank $2$ case, certain Laurent mutation invariants are exactly the greedy bases of classical cluster algebras of affine type, see \cite{LLZ14, CL25}.
\end{remark}
\begin{definition}[\emph{Mutation-preserving generalized cluster algebras}]\label{def: mutation-preserving}
  Let  $\mathbf{\Sigma}_{\mathbf{GCA}}$ be a generalized cluster pattern. For any $t \in \mbT_{n}$ and $k \in \{1,\cdots,n\}$, we can define a map $\hat{\mu}_{k;t}$ from $\mbQ^*(x_1,\cdots,x_n)^{\times n}$ to itself as $\hat{\mu}_{k;t}(f_1,\dots,f_n)=(f^{\prime}_1,\dots,f^{\prime}_n)$, where
	\begin{align}  
		f^{\prime}_i=\left\{
		\begin{array}{ll}
			 f_{k}^{-1}\left(\mathop{\prod}\limits_{j=1}^{n} f_j^{[-b_{jk;t}]_+}\right)^{r_k}Z_k\left(\mathop{\prod}\limits_{j=1}^{n} f_j^{b_{jk;t}}\right), &   i=k, \\
			f_i, &   i \neq k, 
		\end{array} \right. \label{eq: mutation-preserving}
	\end{align}
	where $b_{jk;t}$ is the $(j,k)$-component of $B_t$. Note that $\hat{\mu}_{k;t}$ is an involution. If $((f_1,\cdots,f_n),B_t)$ is a generalized seed of  $\mcA_{\mathbf{GCA}}$, then $\hat{\mu}_{k;t}$ can be viewed as the generalized cluster mutation in the $k$-th direction.
	
	A \emph{mutation-preserving generalized cluster algebra} is a generalized cluster algebra $\mcA_{\mathbf{GCA}}$ which satisfies
	$\hat{\mu}_{k;t}=\hat{\mu}_{k;t_0},$ for any  $t \in \mbT_{n}$ and $k \in \{1,\cdots,n\}.$
\end{definition}
%Given a mutation-preserving generalized cluster algebra, we denote $\hat{\mu}_{k;t}$ by $\hat{\mu}_{k}$ for any  $t \in \mbT_{n}$ and $k \in \{1,\cdots,n\}.$ 
Then, we can directly obtain the following lemma.
\begin{lemma}\label{equivalent} Let $\mcA_{\mathbf{GCA}}$ be a mutation-preserving generalized cluster algebra and $\mcT(x_{1},\cdots,x_{n})$ be a non-constant rational function. Then, the following statements are equivalent.
\begin{enumerate}
	\item $\mcT(x_{1},\cdots,x_{n})$ is a generalized mutation invariant of $\mcA_{\mathbf{GCA}}$.
	\item $\mcT(x_{1},\cdots,x_{n})=\mcT(\hat{\mu}_{k;t_0}(x_{1},\cdots,x_{n}))$, for any $k\in \{1,\dots,n\}$.
\end{enumerate}
\end{lemma}
In fact, if $B$ is sign-equivalent, it may not be difficult to prove that $\mcA_{\mathbf{GCA}}$ is mutation-preserving. Conversely, if $\mcA_{\mathbf{GCA}}$ is mutation-preserving, does the mutation-preserving property imply that $B$ is sign-equivalent? This question is proposed by T. Nakanishi. In the next section, we will give an answer to it.
\begin{definition}[\emph{Generalized cluster mutation group}]\label{def: gmp}
	For any mutation-preserving generalized cluster algebra $\mcA_{\mathbf{GCA}}$, we define the \emph{generalized cluster mutation group} as $\widehat{\Gamma}=\langle \widehat{\mu}_1, \widehat{\mu}_2,\dots \widehat{\mu}_n \rangle$, which is generated by the generalized mutations. 
	
	Here, we omit the label $t\in \mbT$ since each $\widehat{\mu}_i$ is independent of the choice of generalized seeds. In particular, if it degenerates to the classical cluster algebra, the cluster mutation group is denoted by $\Gamma=\langle \mu_1, \mu_2, \dots, \mu_n\rangle$.
\end{definition}
\section{Classification of mutation-preserving generalized cluster algebras}\label{sec: classification}
In this section, we classify all irreducible generalized cluster algebras whose mutation maps are independent of the seed. The proof combines the classification of sign-equivalent exchange matrices with restrictions on exchange degrees and exchange polynomials. 

Beforehand, we first recall the following classification of irreducible sign-equivalent exchange matrices for classical cluster algebras.
\begin{theorem}[{\cite[Theorem 3.5]{CL25}}]\label{thm: sign-equivalent}
	All the irreducible sign-equivalent exchange matrices of order $3$ are as follows:
	\begin{align}
		\sigma\begin{pmatrix}0 & 2 & -2\\ -2 & 0 & 2\\ 2 & -2 & 0\end{pmatrix},\ \sigma\begin{pmatrix}0 & 1 & -1\\ -4 & 0 & 2\\ 4 & -2 & 0\end{pmatrix},\ \sigma\begin{pmatrix}0 & 4 & -4\\ -1 & 0 & 2\\ 1 & -2 & 0\end{pmatrix}, \label{sign-equivalence class}
	\end{align} where $\sigma\in \mathfrak{S}_3$.
\end{theorem}
\begin{remark}
	To some degree, the classification \eqref{sign-equivalence class} also gives an answer to Gyoda's question \cite[Question 20]{GM23} by \Cref{CA GCA mutation}. In addition, the corresponding generalized Cartan matrices $A(B)=(a_{ij})_{3\times 3}$, where $a_{ii}=2\ (i=1,2,3)$ and $a_{ij}=-|b_{ij}|$ for any $i\neq j$, are of \emph{hyperbolic type} in the sense of Kac-Moody algebras \cite{Kac90, HK02}.
\end{remark}
\begin{theorem}\label{thm: classification}
	All the irreducible mutation-preserving generalized cluster algebras $\mcA_{\mathbf{GCA}}=\mcA(B;R;\mcZ)$ are as follows:
	\begin{enumerate}
		\item $\mcA(B;R;\mcZ)$ is of rank $2$ with arbitrary coefficients for $\mcZ$.
		\item $\mcA(B;R;\mcZ)$ is of rank $3$ with the pair $(B;R)$ given in \Cref{tab: rank-three-classification} except (II.2) and arbitrary coefficients for $\mcZ$.
		\item $\mcA(B;R;\mcZ)$ is of rank $3$ with the pair $(B;R)$ given in (II.2) and the coefficients for $\mcZ$  satisfy the reciprocity condition. 
	\end{enumerate}
\end{theorem}
\begin{proof} Firstly, let $\mcA(B;R;\mcZ)$ be any of the cases $(1)$, $(2)$ and $(3)$, which is irreducible.  Then, we have $\hat{\mu}_i(B)=-B$ for any $i$. Let $t_0$ and $t_1$ be any $k$-adjoint vertices. For any $1\leq \beta
\leq r_k$, we have 
\begin{align*}
	r_k[-b_{ik;t_0}]_++\beta b_{ik;t_0}=-r_k[b_{ik;t_0}]_+-(r_k-\beta)b_{ik;t_0}.
\end{align*} It implies that 
\begin{equation}
\begin{aligned}
	r_k[-b_{ik;t_0}]_++\beta b_{ik;t_0}&=r_k[b_{ik;t_1}]_+-\beta b_{ik;t_1}\\ &= r_k[-b_{ik;t_1}]_++(r_k-\beta) b_{ik;t_1}. \label{eq: reverse}
\end{aligned}
\end{equation} Note that $\mcZ$ only contains the exchange polynomials $Z_k$ in the form of $1+z$, $1+\gamma z+z^2$ or $1+\gamma z+ \omega z^2+ \gamma z^4 +z^5$, where $\gamma, \omega \in \mbN$. Hence, based on \eqref{eq: mutation-preserving} and \eqref{eq: reverse}, we have $\hat{\mu}_{k;t_0}=\hat{\mu}_{k;t_1}$. By induction, we conclude that $\mcA_{\mathbf{GCA}}$ is mutation-preserving.

Secondly, assume that $\mcA_{\mathbf{GCA}}$ is an irreducible mutation-preserving generalized cluster algebra. Let $t_0$ and $t_1$ be any $k$-adjoint vertices. Then we have $\hat{\mu}_{k;t_0}=\hat{\mu}_{k;t_1}$. 

	\textbf{Step $1$}: By comparing the exponents of the algebraically independent variables $f_i$ in \eqref{eq: mutation-preserving}, for any $i\in \{1,\dots,n\}$, we have \begin{align*}
		\mfb_{i;t_0}=\mfb_{i;t_1}\ \text{or}\ \mfb_{i;t_0}=-\mfb_{i;t_1}, 
	\end{align*} where $\mfb_{i;t_0}$ and $\mfb_{i;t_1}$ are the $i$-th columns of $B_{t_0}$ and $B_{t_1}$ respectively.
	
	\textbf{Step $2$}: We now recover them to the column vectors of the exchange matrices of the classical cluster algebras. That is to say, by \eqref{eq: recover}, we might assume that $\widebar{\mfb}_{i;t_0}=r_i\mfb_{i;t_0}$ and $\widebar{\mfb}_{i;t_1}=r_i\mfb_{i;t_1}$, where $\widebar{\mfb}_{i;t_0}$ and $\widebar{\mfb}_{i;t_1}$ are the $i$-th columns of $\widebar{B}_{t_0}=B_{t_0}R$ and $\widebar{B}_{t_1}=B_{t_1}R$ respectively. Then, both of the exchange matrices are irreducible and for any $i\in \{1,\dots,n\}$, we have \begin{align*}
		\widebar{\mfb}_{i;t_0}=\widebar{\mfb}_{i;t_1}\ \text{or}\ \widebar{\mfb}_{i;t_0}=-\widebar{\mfb}_{i;t_1}.
	\end{align*}
	
	\textbf{Step $3$}: We claim that for any $i\in \{1,\dots,n\}$, $\widebar{\mfb}_{i;t_0}=-\widebar{\mfb}_{i;t_1}$, which means that $\widebar{B}_{t_0}=-\widebar{B}_{t_1}$. Without loss of generality, up to permutation, we might assume that $k=1$ and $\widebar{b}_{21;t_0}, \dots, \widebar{b}_{l1;t_0}\neq 0$, but $\widebar{b}_{l+1,1;t_0}, \dots, \widebar{b}_{n1;t_0} =0$. This implies that $\widebar{\mfb}_{j;t_0}=-\widebar{\mfb}_{j;t_1}$ for $2\leq j\leq l$. Then, for $l+1\leq h\leq n$ and $2\leq j\leq n$, we have 
	\begin{align*}
		\widebar{b}_{hj;t_1}=\widebar{b}_{hj;t_0}+\sgn(\widebar{b}_{h1;t_0})[\widebar{b}_{h1;t_0}\widebar{b}_{1j;t_0}]_+=\widebar{b}_{hj;t_0}.
	\end{align*} However, for $2\leq j\leq l$, we have $\widebar{b}_{hj;t_1}=-\widebar{b}_{hj;t_0}$. Hence, we obtain that $\widebar{b}_{hj;t_0}=0$ for any $l+1\leq h\leq n$ and $1\leq j\leq l$. This also implies that $\widebar{b}_{hj;t_0}=0$ for any $l+1\leq j\leq n$ and $1\leq h\leq l$. It follows that $\widebar{B}_{t_0}$ is reducible, which is a contradiction. Thus, the claim holds, that is $\mu_k(\widebar{B}_{t_0})=-\widebar{B}_{t_0}$ for any $k\in \{1,\dots,n\}$.
	
	\textbf{Step $4$}: By Step 3 and \Cref{thm: sign-equivalent}, we conclude that $\widebar{B}_{t_0}$ must be one of three types in \eqref{sign-equivalence class}. Then, all the possible pairs $(B_{t_0};R)$ are listed in \Cref{tab: rank-three-classification}. Equivalently, all the possible pairs $(B;R)$ are also of these forms. Note that for $n=2$ and $n=3$ except for case (II.2), the coefficients of $Z_k$ can be arbitrary. However, in case (II.2), the reciprocity condition should hold.
	
	Hence, all the irreducible mutation-preserving generalized cluster algebras are classified.
\end{proof}
\begin{table}[p]
\centering
\renewcommand{\arraystretch}{1.13}
\setlength{\tabcolsep}{8pt}

\begin{tabular}{c c c p{6cm}}
\toprule
\textbf{Case} & \textbf{Matrix $B$} & \textbf{Degree $R$} & \textbf{Description} \\
\midrule

(I.1) &
$\sigma\begin{pmatrix}
0 & 2 & -2\\
-2 & 0 & 2\\
2 & -2 & 0
\end{pmatrix}$ &
$\begin{pmatrix}
1 & 0 & 0\\
0 & 1 & 0\\
0 & 0 & 1
\end{pmatrix}$ &
$BR=\sigma\begin{pmatrix}
0 & 2 & -2\\
-2 & 0 & 2\\
2 & -2 & 0
\end{pmatrix},\ \sigma\in \mathfrak{S}_3$ \\

\addlinespace
(I.2) &
$\sigma\begin{pmatrix}
0 & 2 & -2\\
-1 & 0 & 2\\
1 & -2 & 0
\end{pmatrix}$ &
$\sigma\begin{pmatrix}
2 & 0 & 0\\
0 & 1 & 0\\
0 & 0 & 1
\end{pmatrix}$ &
$BR=\sigma\begin{pmatrix}
0 & 2 & -2\\
-2 & 0 & 2\\
2 & -2 & 0
\end{pmatrix},\ \sigma\in \mathfrak{S}_3$ \\

\addlinespace
(I.3) &
$\sigma\begin{pmatrix}
0 & 1 & -2\\
-1 & 0 & 2\\
1 & -1 & 0
\end{pmatrix}$ &
$\sigma\begin{pmatrix}
2 & 0 & 0\\
0 & 2 & 0\\
0 & 0 & 1
\end{pmatrix}$ &
$BR=\sigma\begin{pmatrix}
0 & 2 & -2\\
-2 & 0 & 2\\
2 & -2 & 0
\end{pmatrix},\ \sigma\in \mathfrak{S}_3$ \\

\addlinespace
(I.4) &
$\sigma\begin{pmatrix}
0 & 1 & -1\\
-1 & 0 & 1\\
1 & -1 & 0
\end{pmatrix}$ &
$\begin{pmatrix}
2 & 0 & 0\\
0 & 2 & 0\\
0 & 0 & 2
\end{pmatrix}$ &
$BR=\sigma\begin{pmatrix}
0 & 2 & -2\\
-2 & 0 & 2\\
2 & -2 & 0
\end{pmatrix},\ \sigma\in \mathfrak{S}_3$ \\

\addlinespace

(II.1) &
$\sigma\begin{pmatrix}
0 & 1 & -1\\
-2 & 0 & 2\\
2 & -2 & 0
\end{pmatrix}$ &
$\sigma\begin{pmatrix}
2 & 0 & 0\\
0 & 1 & 0\\
0 & 0 & 1
\end{pmatrix}$ &
$BR=\sigma\begin{pmatrix}
0 & 1 & -1\\
-4 & 0 & 2\\
4 & -2 & 0
\end{pmatrix},\ \sigma\in \mathfrak{S}_3$ \\

\addlinespace

(II.2) &
$\sigma\begin{pmatrix}
0 & 1 & -1\\
-1 & 0 & 2\\
1 & -2 & 0
\end{pmatrix}$ &
$\sigma\begin{pmatrix}
4 & 0 & 0\\
0 & 1 & 0\\
0 & 0 & 1
\end{pmatrix}$ &
$BR=\sigma\begin{pmatrix}
0 & 1 & -1\\
-4 & 0 & 2\\
4 & -2 & 0
\end{pmatrix},\ \sigma\in \mathfrak{S}_3$ \\
(II.3) &
$\sigma\begin{pmatrix}
0 & 1 & -1\\
-4 & 0 & 2\\
4 & -2 & 0
\end{pmatrix}$ &
$\begin{pmatrix}
1 & 0 & 0\\
0 & 1 & 0\\
0 & 0 & 1
\end{pmatrix}$ &
$BR=\sigma\begin{pmatrix}
0 & 1 & -1\\
-4 & 0 & 2\\
4 & -2 & 0
\end{pmatrix},\ \sigma\in \mathfrak{S}_3$ \\

\addlinespace
(III.1) &
$\sigma\begin{pmatrix}
0 & 4 & -4\\
-1 & 0 & 2\\
1 & -2 & 0
\end{pmatrix}$ &
$\begin{pmatrix}
1 & 0 & 0\\
0 & 1 & 0\\
0 & 0 & 1
\end{pmatrix}$ &
$BR=\sigma\begin{pmatrix}
0 & 4 & -4\\
-1 & 0 & 2\\
1 & -2 & 0
\end{pmatrix},\ \sigma\in \mathfrak{S}_3$ \\

\addlinespace

(III.2) &
$\sigma\begin{pmatrix}
0 & 2 & -4\\
-1 & 0 & 2\\
1 & -1 & 0
\end{pmatrix}$ &
$\sigma\begin{pmatrix}
1 & 0 & 0\\
0 & 2 & 0\\
0 & 0 & 1
\end{pmatrix}$ &
$BR=\sigma\begin{pmatrix}
0 & 4 & -4\\
-1 & 0 & 2\\
1 & -2 & 0
\end{pmatrix},\ \sigma\in \mathfrak{S}_3$ \\

\addlinespace

(III.3) &
$\sigma\begin{pmatrix}
0 & 4 & -2\\
-1 & 0 & 1\\
1 & -2 & 0
\end{pmatrix}$ &
$\sigma\begin{pmatrix}
1 & 0 & 0\\
0 & 1 & 0\\
0 & 0 & 2
\end{pmatrix}$ &
$BR=\sigma\begin{pmatrix}
0 & 4 & -4\\
-1 & 0 & 2\\
1 & -2 & 0
\end{pmatrix},\ \sigma\in \mathfrak{S}_3$ \\

\addlinespace
(III.4) &
$\sigma\begin{pmatrix}
0 & 2 & -2\\
-1 & 0 & 1\\
1 & -1 & 0
\end{pmatrix}$ &
$\sigma\begin{pmatrix}
1 & 0 & 0\\
0 & 2 & 0\\
0 & 0 & 2
\end{pmatrix}$ &
$BR=\sigma\begin{pmatrix}
0 & 4 & -4\\
-1 & 0 & 2\\
1 & -2 & 0
\end{pmatrix},\ \sigma\in \mathfrak{S}_3$ \\

\addlinespace
\bottomrule
\end{tabular}
\vspace{1mm}
\caption{Classification of the pairs $(B;R)$ of rank $3$}
\label{tab: rank-three-classification}
\end{table}
\begin{remark}
	Note that the elements in (III.2) and (III.3) are completely different (up to a sign). However, the corresponding generalized cluster algebras are isomorphic. Moreover, except class (I.1) has $2$ elements, each class has $6$ elements, where the corresponding generalized cluster algebras are isomorphic. Hence, there are $10$ non-isomorphic irreducible mutation-preserving generalized cluster algebras of rank $3$. In addition, classes (I.1), (II.3) and (III.1) are exactly the classical ones in \eqref{sign-equivalence class}.
 \end{remark}
\section{A Markov-type Diophantine equation with multiple orbits}\label{sec: diophantine}
In this section, we apply generalized mutation invariants to a Markov-type Diophantine equation and determine its positive integer solutions. We establish the existence criterion, classify the resulting mutation orbits, and relate the generalized equation to a classical cluster-algebraic counterpart.

In \cite{Pro20}, the relation between the classical Markov equation and the cluster mutation invariant \begin{align*}
	\dfrac{x^2+y^2+z^2}{xyz}
\end{align*} was found, which corresponds to the first exchange matrix of \eqref{sign-equivalence class} or Markov quiver. 
In \cite{CL25}, another Markov-type Diophantine equation corresponding to the cluster mutation invariant \begin{align*}
	\dfrac{x^2+y^4+z^4+2xy^2+2xz^2}{xy^2z^2}
\end{align*} given by \cite{Lam16} was studied. Their mutations come from the classical cluster algebras, whose exchange matrices $B$ are the first and second matrices of \eqref{sign-equivalence class}, respectively.   
\begin{theorem}[{\cite{Hur07},\cite{Lam16}, \cite{CL25}}] Let $k\in \mbN$. Then, the following statements hold. 
	\begin{enumerate}
		\item $x^2+y^2+z^2=kxyz$ has positive integer solutions if and only if $k=1,3$.
		\item $x^2+y^4+z^4+2xy^2+2xz^2=kxy^2z^2$ has positive integer solutions if and only if $k=7$.
	\end{enumerate} Moreover, all the positive integer solutions can be generated by some initial solution via finite cluster mutations.
\end{theorem}

Now, we focus on the classical cluster algebra $\mcA(B_0)$, whose initial exchange matrix is the third type of \eqref{sign-equivalence class}: \begin{align}
	B_0=\begin{pmatrix}0 & -4 & 4\\ 1 & 0 & -2\\ -1 & 2 & 0\end{pmatrix}.\label{4-matrix}
\end{align} Note that $[B_0]=\{B_0, -B_0\}$ and the cluster mutations are given by  
\begin{equation*}
%\begin{aligned}
%	\mu_{1}(x_1,x_2,x_3)&=\left(\dfrac{x_2+x_3}{x_1},x_2,x_3\right),\ \mu_{2}(x_1,x_2,x_3)=\left(x_1,\dfrac{x_1^4+x_3^2}{x_2},x_3\right),\\ \mu_{3}(x_1,x_2,x_3)&=\left(x_1,x_2,\dfrac{x_1^4+x_2^2}{x_3}\right).
%\end{aligned}
\begin{aligned}
	\mu_{1}(x,y,z)&=\left(\dfrac{y+z}{x},y,z\right),\ \mu_{2}(x,y,z)=\left(x,\dfrac{x^4+z^2}{y},z\right),\\ \mu_{3}(x,y,z)&=\left(x,y,\dfrac{x^4+y^2}{z}\right).
\end{aligned}
\end{equation*} Since only three variables are involved, we denote them by $\{x,y,z\}$ rather than usual cluster notations $\{x_1,x_2,x_3\}$.
 Although the second and third exchange matrices of \eqref{sign-equivalence class} are the same up to transposition, their mutation invariants are completely different. In \cite{Kau24}, a classical cluster mutation invariant of $\mcA(B_0)$ is given as 
\begin{align*}\mcT(x,y,z)=\dfrac{x^4+y^2+z^2+2yz}{x^2yz}.\end{align*} Then, we aim to study the new corresponding Markov-type Diophantine equation \begin{align*} x^4+y^2+z^2+2yz=kx^2yz, \end{align*} where $k\in \mbN$.
 For later use, we first briefly recall some basic proofs and properties.
%$$F_k(x,y,z)=x^4+y^2+z^2+2yz-kx^2yz$$
\begin{lemma}
	$\mcT(x,y,z)$ is a Laurent mutation invariant of $\mcA(B_0)$. 
\end{lemma}
\begin{proof} We only need to prove that $\mcT(\mu_i(x,y,z))=\mcT(x,y,z)$ for $i=1,2,3$. Let $\mcT=\mcT(x,y,z)$. Firstly, if $i=1$, then $x$ is a zero point of the biquadratic polynomial $f(\lambda)=\lambda^4-\mcT yz\lambda^2+(y+z)^2$. Hence, by Vieta's formula, we find that $x^{\prime}=\frac{y+z}{x}$ is another zero point, which implies that $\mcT(\mu_1(x,y,z))=\mcT(x,y,z)$. 

Secondly, if $i=2$, then $y$ is a zero point of the quadratic polynomial $g(\lambda)=\lambda^2+(2z-\mcT x^2z)\lambda+x^4+z^2$. Then, by Vieta's formula, $y^{\prime}=\frac{x^4+y^2}{y}$ is its another zero point. Hence, we have $\mcT(\mu_2(x,y,z))=\mcT(x,y,z)$, which also implies that $\mcT(\mu_3(x,y,z))=\mcT(x,y,z)$ by symmetry.
%Since $\mu_1(X,Y,Z)=(\dfrac{Y+Z}{X},Y,Z)$, we have \begin{align*}M\circ \mu_1(X,Y,Z)&=\frac{(\frac{Y+Z}{X})^4+Y^2+Z^2+2YZ}{(\frac{Y+Z}{X})^2YZ}\\ &=\frac{(Y+Z)^4+X^4(Y+Z)^2}{X^2(Y+Z)^2YZ}\\ &=\frac{(Y+Z)^2+X^2}{X^2YZ}\\ &=M(X,Y,Z).\end{align*} Similarly, we also have $M\circ \mu_i(X,Y,Z)=M(X,Y,Z),i=2,3.$ 
Thus, $\mcT(x,y,z)$ is a mutation invariant of $\mcA(B_0)$. 
\end{proof}
\begin{remark}\label{lem: integrality} In the proof, the mutation rules $yy^{\prime}=x^4+z^2$ and $zz^{\prime}=x^4+y^2$ can be regarded as Vieta's formulas. Note that $y+y^{\prime}=\mcT x^2z-2z$ and $z+z^{\prime}=\mcT x^2y-2y$. Hence, if $(x,y,z,\mcT)\in \mbN^4$, then so are $y^{\prime}$ and $z^{\prime}$. In addition, it also implies that ${x^{\prime}}^2=\mcT yz-x^2 = \frac{(y+z)^2}{x^2}\in \mbN$. Hence, $x^{\prime}$ is also a positive integer.
\end{remark}

In fact, we may consider a variant of cluster mutation maps $\widehat{\mu}_i: \mbQ_+^3\rightarrow \mbQ_+^3$ given by \begin{equation}
\begin{aligned}
	\widehat{\mu}_1(x,y,z)&=\left(\dfrac{(y+z)^2}{x},y,z\right),\
\widehat{\mu}_2(x,y,z)=\left(x,\dfrac{x^2+z^2}{y},z\right),\\
\widehat{\mu}_3(x,y,z)&=\left(x,y,\dfrac{x^2+y^2}{z}\right).
\end{aligned}\label{eq: generalized mutation rules}
\end{equation} Note that they correspond to the generalized cluster mutations \eqref{eq: generalized cluster variables} associated with 
\begin{align*}
	B=\begin{pmatrix}0 & -2 & 2\\ 1 & 0 & -2\\ -1 & 2 & 0\end{pmatrix},\ R=\begin{pmatrix}2 & 0 & 0\\ 0 & 1 & 0\\ 0 & 0 & 1\end{pmatrix}.
\end{align*} Similarly, we may directly check that the following Laurent polynomial is a generalized mutation invariant, where the integrality is also preserved as \Cref{lem: integrality}.
\begin{align*}
	\widehat\mcT(x,y,z)=\dfrac{x^2+y^2+z^2+2yz}{xyz}.
\end{align*}

Let $S$ be a map from $\mbQ_+^3$ to $\mbQ_+^3$ defined by $S(x,y,z)=(x^2,y,z)$. Then, we have $\mu_i$ and $\widehat{\mu}_i$ are \emph{compatible} with $S$ as follows. 
\begin{lemma}\label{contract} For any $i=1,2,3$, the equality holds:
\begin{align*}\widehat{\mu}_i(S(x,y,z))=S(\mu_i(x,y,z)).\end{align*}
\end{lemma}
\begin{proof}
This follows by direct calculation. 
\end{proof}

\begin{proposition}\label{prop: k larger than 5}
	Let $k\in \mbN$ and the Diophantine equation be as follows:\begin{align*}
		x^2+y^2+z^2+2yz=kxyz.
	\end{align*} Then, if $k\geq 6$, there is no positive integer solution.
\end{proposition}
\begin{proof} 
%If $k=1$, the equation has a positive integer solution $(9,6,3)$.
%If $k=2$, the equation has a positive integer solution $(4,2,2)$.
%If $k=3$, the equation has a positive integer solution $(3,2,1)$.
%If $k=4$, the equation has a positive integer solution $(2,1,5)$.
%If $k=5$, the equation has a positive integer solution $(1,1,1)$.
Assume that $(A,B,C)$ is a positive integer solution and $B\geq C$ by symmetry. We then claim that after mutating a solution at the direction where the component is maximal, we can obtain a new solution whose maximal component is strictly decreasing. If $A\geq B$, let $G_k(\lambda)=\lambda^2-kBC\lambda+(B+C)^2$. Then we have \begin{align*}G_k(B)=2B^2+C^2+2BC-kB^2C\leq 5B^2-6B^2=-B^2<0.\end{align*} Thus, $\max(\widehat{\mu}_1(A,B,C))=B< A$. If $B\geq A \geq C$, let $H_k(\lambda)=\lambda ^2+(2C-kAC) \lambda+A^2+C^2$. Then, we obtain that 
\begin{align*}H_k(A)=A^2+2CA-kA^2C+A^2+C^2 \leq 5A^2-6A^2=-A^2<0.\end{align*} Thus, we have $\max(\widehat{\mu}_2(A,B,C))=A< B$. If $B\geq C \geq A$, then \begin{align*}H_k(C)=2C^2+A^2+2AC-kAC^2 \leq 5C^2-6C^2=-C^2<0.\end{align*} Hence, it implies that $\max(\widehat{\mu}_2(A,B,C))=C< B$. Since such a process must terminate due to the lower bound, we arrive at a contradiction.\end{proof}
For later use, we denote the above two \emph{discriminant polynomials} by \begin{equation*}
\begin{aligned}
	\text{(I)}\ G_k(\lambda)&=\lambda^2-kBC\lambda+(B+C)^2,\\ \text{(II)}\ H_k(\lambda)&=\lambda ^2+(2C-kAC) \lambda+A^2+C^2. 
\end{aligned}
\end{equation*}
%\begin{lemma}
%	Let $k\in \mbN$ and the Diophantine equation be as follows:
%	\begin{align*}
%	X^2+Y^2+Z^2+2YZ=kXYZ. 
%	\end{align*}
%	If it has positive integer solutions, then there exists a positive integer solution $(a,b,c)$ such that one of $G_k(b)$, $G_k(c)$ and $H_k(a)\geq 0$.
%\end{lemma}
%\begin{proof}
%\textcolor{red}{Is it necessary? whether we can delete it? I am not sure}
%\end{proof}

\begin{theorem}\label{thm: main}
	For any $1\leq k\leq 5$, all the positive integer solutions to the Diophantine equation \begin{align}x^2+y^2+z^2+2yz=kxyz\label{equ: CH}\end{align} can be obtained by some initial solution by finite generalized cluster mutations $\widehat{\mu}_i$ in $\widehat{\Gamma}=\langle \widehat{\mu}_1, \widehat{\mu}_2, \widehat{\mu}_3\rangle$. More precisely,
	\begin{enumerate}
		\item For $k=1$, there are exactly four orbits of solutions: $\widehat{\Gamma}[(9,6,3)]$, $\widehat{\Gamma}[(9,3,6)]$, $\widehat{\Gamma}[(8,4,4)]$ and $\widehat{\Gamma}[(5,5,5)]$;
		\item For $k=2$, there is only one orbit of solutions: $\widehat{\Gamma}[(4,2,2)]$;
		\item For $k=3$, there are exactly two orbits of solutions: $\widehat{\Gamma}[(3,2,1)]$ and $\widehat{\Gamma}[(3,1,2)]$; 
		\item For $k=4$, there is only one orbit of solutions: $\widehat{\Gamma}[(2,1,1)]$;
		\item For $k=5$, there is only one orbit of solutions: $\widehat{\Gamma}[(1,1,1)]$.
	\end{enumerate}
\end{theorem}
\begin{proof} We distinguish five cases according to the value of $k$. The case $k=5$ follows directly from \cite[Theorem 1]{GM23}. In the following, we discuss the cases that $1\leq k\leq 4$.
\begin{enumerate}
\item When $k=1$, we assume that $(A,B,C)$ is a solution distinct from $$(9,6,3), (9,3,6), (8,4,4),(5,5,5),(20,5,5),(5,5,10),(5,10,5).$$ By symmetry, we may assume that $B\geq C$. \textbf{We claim that after mutating a solution in the direction where the component is maximal, we can obtain a new solution whose maximal component is strictly decreasing}.\\
$(1.1)$ Assume that $A \geq B \geq C$ and we mutate at $\widehat{\mu}_1$. If $G_1(B)<0$, we obtain that $\max(\widehat{\mu}_1(A,B,C))=B< A=\max(A,B,C)$ and the claim holds. If $G_1(B) \geq 0$, we have 
\begin{align*}
5B^2 - CB^2 \geq 2B^2 + C^2 + 2BC - B^2C \geq 0,
\end{align*}
from which we conclude that $C \leq 5$. On the other hand, from the equation $ABC = A^2 + (B + C)^2$ and the inequality $A^2 + (B + C)^2 \geq 2A(B + C)$, we deduce that $C \geq 3$. Therefore, we obtain that $
3 \leq C \leq 5.$
If $C=5$, $G_1(B)\geq 0$ implies $B=5$, and then $(A,B,C)=(5,5,5)$ or $(20,5,5)$, contradicting our assumption. If $C=4$, then $G_1(B)\geq 0$ implies $4\leq B\leq 5$. For $B=5$, the equation $A^2-20A+81=0$ has no integer solution, while for $B=4$ we have $(A,B,C)=(8,4,4)$, again contradicting our assumption. If $C=3$, then $G_1(B)\geq 0$ implies $B\leq 7$. For $B\leq 5$, the equation $A^2-3AB+(B+3)^2=0$ has no positive integer solutions since $\Delta=5B^2-24B-36<0$. Hence $6\leq B\leq 7$. For $B=7$, the equation $A^2-21A+100=0$ has no integer solution, while for $B=6$ we obtain $(A,B,C)=(9,6,3)$, contradicting our assumption.
Hence, the claim holds for such case $(1.1)$.\\
$(1.2)$ Assume that $B \geq A \geq C$ and we mutate at $\widehat{\mu}_2$. If $H_1(A)<0$, we obtain that $\max(\widehat{\mu}_2(A,B,C))=A< B=\max(A,B,C)$ and the claim holds. If $H_1(A) \geq 0$, we have \begin{align*}
5A^2 - CA^2 \geq 2A^2 + C^2 + 2AC - A^2C \geq 0,
\end{align*}
which implies that $C \leq 5$. Note that by the equation $ABC = A^2 + (B + C)^2$ and the inequality $A^2 + (B + C)^2 \geq 2A(B + C)$, we deduce that $C \geq 3$. Hence, we have $
3 \leq C \leq 5.$ If $C=5$, then $H_1(A)\geq 0$ implies $A=5$, and by a direct calculation $(A,B,C)=(5,5,5)$ or $(5,10,5)$, contradicting our assumption. If $C=4$, then $H_1(A)\geq 0$ implies $4\leq A\leq 5$. For $A=4$ the equation $B^2-8B+32=0$ has no positive integer solution, and for $A=5$ the equation $B^2-12B+41=0$ has no positive integer solution. If $C=3$, then $H_1(A)\geq 0$ implies $A\leq 7$, but the equation $B^2+(6-3A)B+A^2+9=0$ has no positive integer solution since $\Delta=A(5A-36)<0$.
Hence, the claim holds for such case $(1.2)$.\\
$(1.3)$ Assume that $B \geq C \geq A$ and we mutate at $\widehat{\mu}_2$. If $H_1(C)<0$, we obtain that  $\max(\widehat{\mu}_2(A,B,C))=C< B=\max(A,B,C)$ and the claim holds. If $H_1(C) \geq 0$, we have
\begin{align*}
5C^2 - AC^2 \geq A^2 + 4C^2 - AC^2 \geq 0,
\end{align*}
which implies that $A \leq 5$. On the other hand, if $A \leq 4$, then we have \begin{align*}
	0=H_1(B)=B^2+(2-A)BC+C^2+A^2>0,
\end{align*} which is a contradiction. Now, we only need to consider $A = 5$. Note that 
$
H_1(C) = -C^2 + 25\geq 0,
$
which implies $C = 5$. Hence, from $H_1(B) = 0$, we conclude that $B = 5$, leading to the solution $(A,B,C) = (5,5,5)$. However, this contradicts our initial assumption.
Hence, the claim holds for such case $(1.3)$.

\item When $k=2$, we assume that $(A,B,C)$ is a solution distinct from $(4,2,2)$. By symmetry, we also assume $B\geq C$. \textbf{We claim that after mutating a solution in the direction where the component is maximal, we can obtain a new solution whose maximal component is strictly decreasing.}\\
$(2.1)$ Assume that $A \geq B \geq C$ and we mutate at $\widehat{\mu}_1$. If $G_2(B)<0$, then we obtain that $\max(\widehat{\mu}_1(A,B,C))=B<A$ and the claim holds. If $G_2(B)\geq 0$, then
\begin{align*}
5B^2-2CB^2\geq 2B^2+C^2+2BC-2B^2C\geq 0,
\end{align*}
which implies that $C\leq 2$. If $C=2$, then
$
G_2(B)=-2B^2+4B+4\geq 0,
$
and hence $B=2$. By substituting $B=C=2$ into the equation, we have $(A,B,C)=(4,2,2)$, contradicting our assumption. If $C=1$, then the equation, regarded as a quadratic equation in $A$, has discriminant
$
\Delta=(2B)^2-4(B+1)^2<0,
$
which is impossible. Hence, the claim holds for such case $(2.1)$.\\
$(2.2)$ Assume that $B\geq A\geq C$ and we mutate at $\widehat{\mu}_2$. If $H_2(A)<0$, then we obtain that $\max(\widehat{\mu}_2(A,B,C))=A<B$ and the claim holds. If $H_2(A)\geq 0$, then
\begin{align*}
5A^2-2CA^2\geq 2A^2+C^2+2AC-2A^2C\geq 0,
\end{align*}
which implies $C\leq 2$. If $C=2$, then we have 
$
H_2(A)=-2A^2+4A+4\geq 0,
$
and hence $A=2$. However, the equation for $B$ becomes $B^2-4B+8=0$, which has no integer solution. If $C=1$, then the equation for $B$ has discriminant
$
\Delta=(2-2A)^2-4(A^2+1)=-8A<0,
$
which is impossible. Hence, the claim holds for such case $(2.2)$.\\
$(2.3)$ Assume that $B\geq C\geq A$ and we mutate at $\widehat{\mu}_2$. If $H_2(C)<0$, then we obtain that $\max(\widehat{\mu}_2(A,B,C))=C<B$ and the claim holds. If $H_2(C)\geq 0$, then
\begin{align*}
5C^2-2AC^2\geq A^2+4C^2-2AC^2\geq 0,
\end{align*}
which implies $A\leq 2$. If $A=2$, then
$
H_2(B)=(B-C)^2+4>0,
$
which contradicts $H_2(B)=0$. If $A=1$, then
$
H_2(B)=B^2+C^2+1>0,
$
which is also impossible. Hence, the claim holds for such case $(2.3)$.

\item When $k=3$, we assume that $(A,B,C)$ is a solution distinct from $(3,2,1)$ and $(3,1,2)$. By symmetry, we also assume $B\geq C$. \textbf{We claim that after mutating a solution in the direction where the component is maximal, we can obtain a new solution whose maximal component is strictly decreasing.}\\
$(3.1)$ Assume that $A\geq B\geq C$ and we mutate at $\widehat{\mu}_1$. If $G_3(B)<0$, then we obtain that $\max(\widehat{\mu}_1(A,B,C))=B<A$ and the claim holds. If $G_3(B)\geq 0$, then
\begin{align*}
5B^2-3CB^2\geq 2B^2+C^2+2BC-3B^2C\geq 0,
\end{align*}
which implies $C=1$. Thus, we have 
$
G_3(B)=-B^2+2B+1\geq 0,
$
and hence $B=1$ or $B=2$. If $B=1$, the equation becomes $A^2-3A+4=0$, which has no integer solution. If $B=2$, then we obtain $(A,B,C)=(3,2,1)$, contradicting our assumption. Hence, the claim holds for such case $(3.1)$.\\
$(3.2)$ Assume that $B\geq A\geq C$ and we mutate at $\widehat{\mu}_2$. If $H_3(A)<0$, then we obtain that $\max(\widehat{\mu}_2(A,B,C))=A<B$ and the claim holds. If $H_3(A)\geq 0$, then
\begin{align*}
5A^2-3CA^2\geq 2A^2+C^2+2AC-3A^2C\geq 0,
\end{align*}
which implies $C=1$. Thus, we obtain that 
$
H_3(A)=-A^2+2A+1\geq 0,
$
and hence $A=1$ or $A=2$. If $A=1$, the equation for $B$ is $B^2-B+2=0$ and if $A=2$, it is $B^2-4B+5=0$. Both equations have no integer solution. Hence, the claim holds for such case $(3.2)$.\\
$(3.3)$ Assume that $B\geq C\geq A$ and we mutate at $\widehat{\mu}_2$. If $H_3(C)<0$, then we obtain that $\max(\widehat{\mu}_2(A,B,C))=C<B$ and the claim holds. If $H_3(C)\geq 0$, then
\begin{align*}
5C^2-3AC^2\geq A^2+4C^2-3AC^2\geq 0,
\end{align*}
which implies $A=1$. In this case, we have 
$
H_3(B)=B^2-BC+C^2+1>0,
$
since $B\geq C$. This contradicts $H_3(B)=0$. Hence, the claim holds for such case $(3.3)$.

\item When $k=4$, we assume that $(A,B,C)$ is a solution distinct from $(2,1,1)$. By symmetry, we also assume $B\geq C$. \textbf{We claim that after mutating a solution in the direction where the component is maximal, we can obtain a new solution whose maximal component is strictly decreasing.}\\
$(4.1)$ Assume that $A\geq B\geq C$ and we mutate at $\widehat{\mu}_1$. If $G_4(B)<0$, then we obtain that $\max(\widehat{\mu}_1(A,B,C))=B<A$ and the claim holds. If $G_4(B)\geq 0$, then
\begin{align*}
5B^2-4CB^2\geq 2B^2+C^2+2BC-4B^2C\geq 0,
\end{align*}
which implies $C=1$. Thus, we have 
$
G_4(B)=-2B^2+2B+1\geq 0,
$
and hence $B=1$. By substituting $B=C=1$ into the equation, we have $(A,B,C)=(2,1,1)$, contradicting our assumption. Hence, the claim holds for such case $(4.1)$.\\
$(4.2)$ Assume that $B\geq A\geq C$ and we mutate at $\widehat{\mu}_2$. If $H_4(A)<0$, then we obtain that $\max(\widehat{\mu}_2(A,B,C))=A<B$ and the claim holds. If $H_4(A)\geq 0$, then
\begin{align*}
5A^2-4CA^2\geq 2A^2+C^2+2AC-4A^2C\geq 0,
\end{align*}
which implies $C=1$. Thus, we have 
$
H_4(A)=-2A^2+2A+1\geq 0,
$
and hence $A=1$. The equation for $B$ becomes $B^2-2B+2=0$, which has no integer solution. Hence, the claim holds for such case $(4.2)$.\\
$(4.3)$ Assume that $B\geq C\geq A$ and we mutate at $\widehat{\mu}_2$. If $H_4(C)<0$, then we obtain that $\max(\widehat{\mu}_2(A,B,C))=C<B$ and the claim holds. If $H_4(C)\geq 0$, then
\begin{align*}
5C^2-4AC^2\geq A^2+4C^2-4AC^2\geq 0,
\end{align*}
which implies $A=1$. Then, we obtain that 
$
H_4(B)=(B-C)^2+1>0,
$
contradicting $H_4(B)=0$. Hence, the claim holds for such case $(4.3)$. \end{enumerate}
\par 
In conclusion, starting from any positive integer solution which is not one of the listed initial solutions or the finitely many exceptional neighboring solutions appearing above, repeated mutation in a maximal direction strictly decreases the maximal component. This process must terminate. The terminal possibilities are precisely the initial solutions listed in the statement, while the exceptional neighboring solutions already belong to the corresponding initial orbits. For instance,
\begin{align*}
(20,5,5)=\widehat{\mu}_1(5,5,5),\ (5,10,5)=\widehat{\mu}_2(5,5,5),\ (5,5,10)=\widehat{\mu}_3(5,5,5).
\end{align*}
Therefore, every positive integer solution lies in one of the stated $\widehat{\Gamma}$-orbits.
\end{proof}
\begin{proposition}\label{prop: scaling correspondence}
	For the Diophantine equation
	\begin{align}x^2+y^2+z^2+2yz=kxyz\label{eq: one to one}\end{align} with $1\leq k\leq 5$, the following statements about one-to-one correspondence hold:
	\begin{enumerate}
		\item For $k=1$, $(A_1,A_2,A_3)$ is a solution in $\widehat{\Gamma}[(9,6,3)]$ (or $\widehat{\Gamma}[(9,3,6)]$) if and only if $3$ divides $A_i\ (i=1,2,3)$ and $(\frac{A_1}{3},\frac{A_2}{3},\frac{A_3}{3})$ is a solution in $\widehat{\Gamma}[(3,2,1)]$ (or $\widehat{\Gamma}[(3,1,2)]$) for $k=3$.
		\item For $k=1$, $(A_1,A_2,A_3)$ is a solution in $\widehat{\Gamma}[(8,4,4)]$ if and only if $2$ divides $A_i\ (i=1,2,3)$ and $(\frac{A_1}{2},\frac{A_2}{2},\frac{A_3}{2})$ is a solution for $k=2$, if and only if $4$ divides $A_i\ (i=1,2,3)$ and $(\frac{A_1}{4},\frac{A_2}{4},\frac{A_3}{4})$ is a solution for $k=4$.
		\item For $k=1$, $(A_1,A_2,A_3)$ is a solution in $\widehat{\Gamma}[(5,5,5)]$ if and only if $5$ divides $A_i\ (i=1,2,3)$ and $(\frac{A_1}{5},\frac{A_2}{5},\frac{A_3}{5})$ is a solution for $k=5$.
	\end{enumerate}
\end{proposition}
\begin{proof}
	If $(A_1,A_2,A_3)$ is a solution to \eqref{eq: one to one}, then we have
	\begin{align*}
		(kA_1)^2+(kA_2)^2+(kA_3)^2+2(kA_2)(kA_3)=k^3A_1A_2A_3=(kA_1)(kA_2)(kA_3), 
	\end{align*} which implies that $(kA_1,kA_2,kA_3)$ is a solution to  \begin{align*}x^2+y^2+z^2+2yz=xyz.\end{align*} On the other hand, suppose that $(B_1,B_2,B_3)$ is a solution for $k=1$. Without loss of generality, we might assume that $(B_1,B_2,B_3)\in \widehat{\Gamma}[(9,6,3)]$. Note that by \eqref{eq: generalized mutation rules} and induction, the divisibility condition is preserved under mutations. Hence, we have $3\ |\ B_i$ for $i=1,2,3$ and $(\frac{B_1}{3},\frac{B_2}{3},\frac{B_3}{3})$ is a solution in $\widehat{\Gamma}[(3,2,1)]$ for $k=3$. Similar statements apply to the other cases.
\end{proof}
This proposition is analogous to a result given by Hurwitz \cite{Hur07} for the classical Markov equation; see also \cite[Proposition 2.2]{Aig13}.

In conclusion, based on \Cref{prop: k larger than 5} and \Cref{thm: main}, we have the following theorem.
\begin{theorem}\label{thm: 12345}
	There exist positive integer solutions to the Diophantine equations \begin{align*}x^2+y^2+z^2+2yz=kxyz\end{align*} if and only if $k=1,2,3,4,5$. Moreover, all the positive integer solutions lie in some $\widehat{\Gamma}$-orbit of solutions. 
\end{theorem}
\begin{remark}
In \Cref{table: markov-type equation}, the first well-known Markov-type equation for each $k$ \cite{Mar80,Hur07}, with unique orbit of solutions, has a classical cluster algebra structure that may be regarded as a degeneration of a generalized cluster algebra. The second Markov-type equation \cite{Lam16, CL25}, with only one orbit of solutions, has a non-degenerate generalized cluster algebra structure. The third new Markov-type equation not only has a non-degenerate generalized cluster algebra structure but also has four orbits of solutions for $k=1$ and two orbits of solutions for $k=3$, thereby providing the multiple-orbit phenomenon and answering the second question of \Cref{question}.
\end{remark}

\begin{remark}
	In fact, the Diophantine equation \eqref{equ: CH} can be viewed as the one with four variables, that is $x^2+y^2+z^2+2yz=xyzw$. Hence, to some degree, \Cref{thm: 12345} provides a special method to deal with such Diophantine equation.
\end{remark}

\begin{table}[htbp] 
\centering

\renewcommand{\arraystretch}{1.3}
\setlength{\tabcolsep}{2pt}

\resizebox{0.9\textwidth}{!}{
\begin{tabular}{
|>{\centering\arraybackslash}m{6cm}
|>{\centering\arraybackslash}m{6cm}
|>{\centering\arraybackslash}m{6cm}|
}
\hline

$B;R$
&
Markov-type equations
&
Number of $\widehat{\Gamma}$-orbits of solutions $\mathcal{N}$
\\
\hline

$\begin{pmatrix}
0 & -2 & 2\\
2 & 0 & -2\\
-2 & 2 & 0
\end{pmatrix}$;
$\begin{pmatrix}
1 & 0 & 0\\
0 & 1 & 0\\
0 & 0 & 1
\end{pmatrix}$
& \vspace{17pt}
$\displaystyle x^2+y^2+z^2=kxyz$
& \vspace{17pt}
$k=1,3:\ \mathcal{N}=1$
\\
\hline

$\begin{pmatrix}
0 & -1 & 1\\
2 & 0 & -1\\
-2 & 1 & 0
\end{pmatrix}$;
$\begin{pmatrix}
1 & 0 & 0\\
0 & 2 & 0\\
0 & 0 & 2
\end{pmatrix}$
& \vspace{17pt}
$\displaystyle
x^2+y^2+z^2+2xy+2xz=kxyz$
& \vspace{17pt}
$k=7:\ \mathcal{N}=1$
\\
\hline

$\begin{pmatrix}
0 & -2 & 2\\
1 & 0 & -2\\
-1 & 2 & 0
\end{pmatrix}$;
$\begin{pmatrix}
2 & 0 & 0\\
0 & 1 & 0\\
0 & 0 & 1
\end{pmatrix}$
& \vspace{17pt}
$\displaystyle x^2+y^2+z^2+2yz=kxyz$
&
\begin{tabular}{rcl}
$k$ & $=$ & $1:\ \mathcal{N}=4$\\
$k$ & $=$ & $3:\ \mathcal{N}=2$\\
$k$ & $=$ & $2,4,5:\ \mathcal{N}=1$
\end{tabular}
\\
\hline

\end{tabular}
}

\vspace{1em}
\caption{Fundamental Markov-type equations with a generalized cluster structure}
\label{table: markov-type equation}
\end{table}

\begin{proposition}\label{prop: no square k34}
    	For $k=3$ and $4$, there is no positive integer solution to the Diophantine equation \begin{align}x^4+y^2+z^2+2yz=kx^2yz. \label{Cor equ}\end{align}
\end{proposition}

\begin{proof}
Note that if $(A,B,C)$ is a positive integer solution to the Diophantine equation \eqref{Cor equ}, then $S(A,B,C)=(A^2,B,C)$ is a solution to the Diophantine equation \begin{align}x^2+y^2+z^2+2yz=kxyz. \label{Cor equ2}\end{align} 
Hence, according to \Cref{thm: main}, there exist $i_1,i_2,\dots,i_n \in \{1, 2, 3\}$, such that $$(A_0,B_0,C_0)=\widehat{\mu}_{i_1} \circ \widehat{\mu}_{i_2} \circ \dots \circ \widehat{\mu}_{i_n} (S(A,B,C))$$ is the initial solution to the Diophantine equation \eqref{Cor equ2}. More precisely, we have 
\begin{align}\label{eq: k=3,4}
(A_0,B_0,C_0)=\left\{
		\begin{array}{ll}
			(3,2,1)\ \text{or}\ (3,1,2), &  \text{if}\ k=3, \\
			(2,1,1), &  \text{if}\ k=4. 
		\end{array} \right.
\end{align} Note that  $\widehat{\mu}_i(S(x,y,z))=S(\mu_i(x,y,z))$ by \Cref{contract}, we obtain that  \begin{align*}(A_0,B_0,C_0)=S (\mu_{i_1} \circ \mu_{i_2} \circ \dots \circ \mu_{i_n} (A,B,C)). \end{align*} Thus, the first component of the initial solution is a perfect square. However, by \eqref{eq: k=3,4}, when $k=3,4$, the first component of the initial solution is not a perfect square number. Therefore, there is no positive integer solution to the Diophantine equation \eqref{Cor equ}.
\end{proof}
Therefore, based on \Cref{thm: 12345} and \Cref{prop: no square k34}, we obtain the following theorem.
\begin{theorem}\label{thm: classical markov type}
Let $k\in \mbN$ and the Diophantine equation be as follows:\begin{align*}
		x^4+y^2+z^2+2yz=kx^2 yz. 
	\end{align*} Then, it has positive integer solutions if and only if $k=1,2,5$. Furthermore, all the solutions can be obtained by some initial solution by finite cluster mutations $\mu_i$ in $\Gamma=\langle \mu_1, \mu_2, \mu_3\rangle$. More precisely, 
	\begin{enumerate}
		\item For $k=1$, there are exactly two orbits of solutions: $\Gamma[(3,6,3)]$ and $\Gamma[(3,3,6)]$;
		\item For $k=2$, there is only one orbit of solutions: $\Gamma[(2,2,2)]$;
		\item For $k=5$, there is only one orbit of solutions: $\Gamma[(1,1,1)]$.
	\end{enumerate}
\end{theorem}

\begin{proof}
By \Cref{thm: 12345,prop: no square k34}, the equation has positive integer solutions if and only if $k=1,2,5$. Let $(A,B,C)$ be a positive integer solution. Then $S(A,B,C)=(A^2,B,C)$ is a positive integer solution to
\begin{align*}
	x^2+y^2+z^2+2yz=kxyz.
\end{align*}
By \Cref{thm: main}, it can be transformed by generalized cluster mutations into one of the initial solutions listed there. Moreover, \Cref{contract} shows that the image of $S$ is preserved under these mutations and that every such mutation sequence lifts to a sequence of classical cluster mutations.

For $k=1$, among the four initial solutions  $(9,6,3)$, $(9,3,6)$, $(8,4,4)$, and $(5,5,5)$, only $(9,6,3)$ and $(9,3,6)$ lie in the image of $S$. For $k=2$ and $k=5$, the corresponding initial solutions are $(4,2,2)$ and $(1,1,1)$, respectively. Taking square roots of their first components gives $(3,6,3)$, $(3,3,6)$, $(2,2,2)$, and $(1,1,1)$. Hence, the positive integer solutions form precisely the classical cluster mutation orbits $\Gamma[(3,6,3)]$, $\Gamma[(3,3,6)]$, $\Gamma[(2,2,2)]$ and $\Gamma[(1,1,1)]$.
\end{proof}
Hence, we can unify \Cref{thm: classical markov type}  with previous results of \cite{Hur07, Lam16, CL25} for the rank $3$ mutation-preserving classical cluster algebras into \Cref{tab: classical}.

\begin{table}[htbp]
\centering

\renewcommand{\arraystretch}{1.3}
\setlength{\tabcolsep}{2pt}

\resizebox{0.9\textwidth}{!}{
\begin{tabular}{
|>{\centering\arraybackslash}m{5cm}
|>{\centering\arraybackslash}m{6.5cm}
|>{\centering\arraybackslash}m{6cm}|
}
\hline

$B$
&
Markov-type equations
&
Number of $\Gamma$-orbits of solutions $\mathcal{N}$
\\
\hline

$\begin{pmatrix}
0 & -2 & 2\\
2 & 0 & -2\\
-2 & 2 & 0
\end{pmatrix}$
& \vspace{17pt}
$\displaystyle x^2+y^2+z^2=kxyz$
& \vspace{17pt}
$k=1,3:\ \mathcal{N}=1$
\\
\hline

$\begin{pmatrix}
0 & -1 & 1\\
4 & 0 & -2\\
-4 & 2 & 0
\end{pmatrix}$
& \vspace{17pt}
 $\displaystyle
x^2+y^4+z^4+2xy^2+2xz^2
=kxy^2z^2$
& \vspace{17pt}
$k=7:\ \mathcal{N}=1$
\\
\hline

$\begin{pmatrix}
0 & -4 & 4\\
1 & 0 & -2\\
-1 & 2 & 0
\end{pmatrix}$
& \vspace{17pt}
$\displaystyle x^4+y^2+z^2+2yz=kx^2yz$
& \vspace{10pt}
\begin{tabular}{rcl}
$k$ & $=$ & $1:\ \mathcal{N}=2$\\
$k$ & $=$ & $2,5:\ \mathcal{N}=1$
\end{tabular}
\\
\hline

\end{tabular}
}

\vspace{1em}
\caption{Fundamental Markov-type equations with a classical cluster structure}
\label{tab: classical}
\end{table}

%=======================================
\section{Generalized Markov Laurent mutation invariants}
\label{sec:laurent-classification}
In this section, we classify all generalized Markov (Laurent) mutation invariants, which correspond to the mutation-preserving generalized cluster algebra of type (I.4) in
Table~\ref{tab: rank-three-classification}.

Three equations in \Cref{table: markov-type equation} corresponding to $k=3,7,5$, respectively, all belong to a family of 
\emph{generalized Markov equations} studied in
\cite{GM23}. Namely,
\begin{align*}
x^2+y^2+z^2+k_1yz+k_2xz+k_3xy
=(3+k_1+k_2+k_3)xyz,
\end{align*}
where $k_1,k_2,k_3\in\mbN$. This leads to the following Laurent
polynomial
\begin{align}
\widehat\mcT=\widehat{\mcT}_{k_1,k_2,k_3}=\frac{x^2+y^2+z^2+k_1yz+k_2xz+k_3xy}{xyz}.
\label{eq:I4-basic-invariant}
\end{align}
In fact, it is a generalized Laurent mutation invariant associated with the mutation-preserving generalized cluster algebra of type (I.4) in
Table~\ref{tab: rank-three-classification}. The associated generalized mutations are given by
\begin{align}
\widehat\mu_1(x,y,z)&=
\left(\frac{y^2+k_1yz+z^2}{x},y,z\right),\
\widehat\mu_2(x,y,z)=
\left(x,\frac{x^2+k_2xz+z^2}{y},z\right),\notag\\
\widehat\mu_3(x,y,z)&=
\left(x,y,\frac{x^2+k_3xy+y^2}{z}\right).
\label{eq:I4-generalized-mutations} 
\end{align} We call it a \emph{generalized Markov Laurent mutation invariant}. Our objective is to classify all generalized Markov Laurent mutation invariants.

We begin by recalling the terminology and preliminary results from field
theory and commutative algebra needed for the proof. Throughout this section, all fields are assumed to have characteristic
zero. Let $L/K$ be a finitely generated field extension.
\begin{enumerate}
    \item A subset $\{u_1,\dots,u_d\}\subseteq L$ is a
    \emph{transcendence basis} of $L/K$ if its elements are
    algebraically independent over $K$ and $L$ is algebraic over
    $K(u_1,\dots,u_d)$. Its cardinality is the \emph{transcendence
    degree} of $L/K$, denoted by $\operatorname{trdeg}_K L$. A finitely
    generated field extension of transcendence degree one is called a
    \emph{function field of one variable} over $K$.

    \item The extension $L/K$ is called \emph{regular} if $K$ is
    algebraically closed in $L$, that is, every element of $L$ that is
    algebraic over $K$ belongs to $K$. Indeed, in characteristic
    zero every finitely generated field extension is separably
    generated, so the separability condition in the general definition
    of a regular extension is automatic.

    \item A nonzero commutative ring $A$ is an \emph{integral domain}
    if it has no zero divisors. Equivalently, $ab=0$ for $a,b\in A$
    implies that $a=0$ or $b=0$.

    \item The notation $K(Y)^\times=K(Y)\setminus\{0\}$ denotes the
    multiplicative group of nonzero rational functions in $Y$ over
    $K$. If $p(Y)\in K[Y]$ is irreducible and
    $r(Y)\in K(Y)^\times$, write $r=p^m a/b$, where
    $a,b\in K[Y]$ are not divisible by $p$. The integer $m$ is the
    \emph{order of $r$ along $p$}, denoted by
    $\operatorname{ord}_p(r)$.

    \item The notation $\operatorname{Aut}_K(E)$ denotes the group of
    field automorphisms of $E$ that fix $K$ pointwise. Thus,
    $\sigma\in\operatorname{Aut}_K(E)$ satisfies $\sigma(a)=a$ for
    every $a\in K$. Such an automorphism has \emph{infinite order} if
    $\sigma^m\neq\operatorname{id}_E$ for every integer $m\geq1$.
    For a subgroup $G\leq\operatorname{Aut}_K(E)$, its \emph{fixed
    subfield} is $E^G\coloneqq\{f\in E\mid\tau(f)=f\text{ for every
    }\tau\in G\}$.
    In particular, the fixed subfield of a single automorphism is denoted by
    $
    E^{\langle\sigma\rangle}
    \coloneqq\{f\in E\mid \sigma(f)=f\}.
    $

    \item An $R$-module $M$ is called \emph{finite free} if it has a
    finite basis over $R$. Equivalently, $M\cong R^d$ for some nonnegative
    integer $d$. Thus, every element of $M$ can be written uniquely as
    an $R$-linear combination of finitely many basis elements.
\end{enumerate}

We shall use the mutation groups above also as groups of field
automorphisms. More precisely, each $\gamma\in\widehat\Gamma$, initially
a birational transformation of the triple $(x,y,z)$, acts from the left
on $\mcF$ by
\begin{align*}
\rho(\gamma)(f)=f(\gamma^{-1}(x,y,z)),\ f\in\mcF.
\end{align*}
The map $\rho:\widehat\Gamma\to\operatorname{Aut}_{\mathbb Q}(\mcF)$
is an injective group homomorphism. We therefore identify
$\widehat\Gamma$ with its image and use the same symbol $\gamma$ for
the birational transformation and the induced field automorphism. In
particular, since every mutation is an involution, we have $\widehat\mu_i(f)=f\bigl(\widehat\mu_i(x,y,z)\bigr)$ and
\begin{align*}
 \mcF^{\widehat\Gamma}=\{f\in\mcF\mid f(\widehat\mu_i(x,y,z))=f,\ i=1,2,3\}.
\end{align*}
The same convention applies to classical cluster mutation groups.

\subsection{Regular function fields}

We first establish the regularity properties of the function fields
associated with the three basic mutation invariants.

\begin{lemma}[Regularity criterion, {\cite{GD65}}]
\label{lem:regularity-criterion}
Let $K$ be a field of characteristic zero, let $L/K$ be a finitely
generated field extension, and let $\overline{K}$ be an algebraic
closure of $K$. Then $L/K$ is regular if and only if $
L\otimes_K \overline{K}$
is an integral domain.
\end{lemma}

\begin{lemma}
\label{lem:base-change-rational-function-field}
Let $K'/K$ be an algebraic field extension and let $Y$ be an
indeterminate. Then there is a natural isomorphism
\begin{align*}
K(Y)\otimes_K K' \cong K'(Y).
\end{align*}
\end{lemma}

\begin{proof}
Let $
S=K[Y]\setminus\{0\}.$
Since localization commutes with scalar extension, we have
\begin{align*}
K(Y)\otimes_K K'
=
S^{-1}K[Y]\otimes_K K'
\cong
S^{-1}K'[Y].
\end{align*}
It remains to prove that every nonzero polynomial in $K'[Y]$ is
invertible in $S^{-1}K'[Y]$.

Let $g(Y)\in K'[Y]$ be nonzero. All the coefficients of $g(Y)$ belong
to some finite extension $K_0/K$ contained in $K'$. Let
$d=[K_0:K]$ and choose a $K$-basis $e_1,\dots,e_d$ of $K_0$. Then,
\begin{align*}
K_0[Y]=K[Y]e_1\oplus\cdots\oplus K[Y]e_d,
\end{align*}
which is a finite free $K[Y]$-module of rank $d$. Multiplication
by $g(Y)$ defines a $K[Y]$-linear map
\begin{align*}
m_g:K_0[Y]\longrightarrow K_0[Y],
\ u\longmapsto g(Y)u, 
\end{align*}
represented in the above basis by a $d\times d$ matrix over $K[Y]$.
Since $K_0[Y]$ is an integral domain and $g(Y)\neq0$, the map $m_g$ is
injective, hence its determinant is nonzero.
More precisely, let $M_g\in \operatorname{Mat}_{d\times d}(K[Y])$ be the matrix of $m_g$ in the basis $e_1,\dots,e_d$.  The
standard adjugate matrix formula gives
\begin{align*}
O \neq M_g\operatorname{adj}(M_g)
=\det(M_g)I_d.
\end{align*}
Let $v_1\in K[Y]^d$ be the coordinate vector of
$1\in K_0[Y]$, and let $h(Y)\in K_0[Y]$ be the element whose
coordinate vector is $\operatorname{adj}(M_g)v_1$. Applying the above
matrix identity to $v_1$ gives
\begin{align*}
g(Y)h(Y)=\det(M_g).
\end{align*}
Since the right hand side belongs to $S$, the polynomial $g(Y)$ is
invertible in $S^{-1}K'[Y]$. Hence,
\begin{align*}
S^{-1}K'[Y]=K'(Y),
\end{align*}
which proves the assertion.
\end{proof}

\begin{example}
Let $K=\mathbb Q$ and $L=\mathbb Q(t)$, where $t$ is transcendental
over $\mathbb Q$. We denote $\overline{\mathbb Q}\subset\mathbb C$
denote the algebraic closure of $\mathbb Q$ inside $\mathbb C$. Then,
 $L/K$ is regular and
$
L\otimes_{\mathbb Q}\overline{\mathbb Q}
\cong\overline{\mathbb Q}(t)
$
is a field, hence an integral domain. In contrast, if
$L'=\mathbb Q(\sqrt2)(t)$, then $L'/\mathbb Q$ is not regular and
\begin{align*}
L'\otimes_{\mathbb Q}\overline{\mathbb Q}
\cong
\overline{\mathbb Q}(t)\times\overline{\mathbb Q}(t),
\end{align*}
which is not an integral domain.
\end{example}

\begin{lemma}
\label{lem:square-even-valuations}
Let $K$ be a field and let $r(Y)\in K(Y)^\times$. If $r(Y)$ is a
square in $K(Y)$, then
$
\operatorname{ord}_{p}(r)
$
is even for every irreducible polynomial $p(Y)\in K[Y]$.
\end{lemma}

\begin{proof}
Write
\begin{align*}
r(Y)=\left(\frac{a(Y)}{b(Y)}\right)^2
\end{align*}
with $a(Y),b(Y)\in K[Y]$ nonzero. For every irreducible polynomial
$p(Y)\in K[Y]$, we have
\begin{align*}
\operatorname{ord}_{p}(r)
=
2\operatorname{ord}_{p}(a)
-
2\operatorname{ord}_{p}(b),
\end{align*}
which is even.
\end{proof}

\begin{proposition}
\label{prop:I4-regular-extension}
Let $t=\widehat\mcT$ and $K=\mathbb Q(t,z)$. Then $\mcF/K$ is a
regular function field of transcendence degree one.
\end{proposition}

\begin{proof}
Regarding $x$ as an algebraic element over $K(y)$, we obtain its annihilating polynomial
\begin{align}
f(X)=X^2+(k_2z+k_3y-tyz)X+y^2+k_1yz+z^2.
\label{eq:I4-quadratic-model}
\end{align}
Hence, $\mcF=\mathbb Q(x,y,z)$ is algebraic over
$\mathbb Q(t,y,z)$, which implies that $t,y,z$ are algebraically
independent and $\operatorname{trdeg}_K\mcF=1$.

Let $\overline K$ be an algebraic closure of $K$. The discriminant of
$f(X)$ about $y$ is given as 
\begin{align*}
\Delta(y)
={}&\bigl((k_3-tz)y+k_2z\bigr)^2
-4(y^2+k_1yz+z^2)\\
={}&\bigl((k_3-tz)^2-4\bigr)y^2
+2z\bigl(k_2(k_3-tz)-2k_1\bigr)y
+(k_2^2-4)z^2.
\end{align*}
Then, the discriminant of this quadratic polynomial $\Delta(y)$ is
\begin{align*}
D=16z^2\bigl((k_3-tz)^2-k_1k_2(k_3-tz)
+k_1^2+k_2^2-4\bigr),
\end{align*}
which is nonzero since $t$ is transcendental over $\mathbb Q(z)$.
Moreover, we have $(k_3-tz)^2-4\neq 0$. Take
\begin{align*}
Y=y+\frac{z\bigl(k_2(k_3-tz)-2k_1\bigr)}{(k_3-tz)^2-4}.
\end{align*}
Then, completing the square gives
\begin{align*}
\Delta(y)
=\bigl((k_3-tz)^2-4\bigr)Y^2
-\frac{D}{4\bigl((k_3-tz)^2-4\bigr)}.
\end{align*}
Over $\overline K$, choose $c\neq0$ such that
\begin{align*}
c^2=\frac{D}{4\bigl((k_3-tz)^2-4\bigr)^2}.
\end{align*}
Then, we have 
$
\Delta(y)=\bigl((k_3-tz)^2-4\bigr)(Y-c)(Y+c).
$
Since the characteristic is zero, $c$ and $-c$ are distinct. Hence,
both roots are simple, and
$
\operatorname{ord}_{Y-c}(\Delta)
=\operatorname{ord}_{Y+c}(\Delta)=1.
$
Thus, Lemma~\ref{lem:square-even-valuations} shows that $\Delta(y)$ is
not a square in $\overline K(y)=\overline K(Y)$. Note that a quadratic
polynomial over a field of characteristic different from two is
reducible if and only if its discriminant is a square. It follows that
$f(X)$ is irreducible over $\overline K(y)$, and hence also over
$K(y)$. The evaluation
homomorphism
\begin{align*}
\Phi:K(y)[X]\longrightarrow\mcF,
\ X\longmapsto x,
\end{align*}
has image $K(y)[x]=K(y)(x)=\mcF$ and kernel $(f(X))$, since $x$ is
algebraic over $K(y)$ and $f(X)$ is its minimal polynomial. Thus, by
the first isomorphism theorem,
\begin{align*}
\mcF\cong K(y)[X]/(f(X)).
\end{align*}
Note that extending scalars from $K$ to $\overline K$ gives
\begin{align*}
\mcF\otimes_K\overline K
&\cong
\left(K(y)[X]/(f(X))\right)\otimes_K\overline K\notag\\
&\cong
\left(K(y)\otimes_K\overline K\right)[X]/(f(X))\notag\\
&\cong
\overline K(y)[X]/(f(X)).
\end{align*}
For the second isomorphism, scalar extension commutes with forming a
polynomial ring. Moreover, $\overline K$ is flat over the field $K$,
so tensoring with $\overline K$ also commutes with taking the quotient
by the ideal $(f(X))$. The last isomorphism follows from
Lemma~\ref{lem:base-change-rational-function-field}. Here, the same
symbol $f(X)$ denotes the image of $f(X)$ after extending its
coefficients from $K$ to $\overline K$. Since $f(X)$ is irreducible
over $\overline K(y)$, the last quotient is a field. The regularity
criterion Lemma~\ref{lem:regularity-criterion} now proves that
$\mcF/K$ is regular.
\end{proof}

\subsection{An infinite-order element}

We exhibit an infinite-order element of the generalized mutation group $\widehat\Gamma=\langle\widehat\mu_1,\widehat\mu_2,
\widehat\mu_3\rangle$, which
will be used to determine the corresponding fixed subfield.

\begin{lemma}
\label{lem:I4-infinite-order}
The element $\sigma=\widehat\mu_1\widehat\mu_2\in \widehat\Gamma$ has infinite
order. In particular, the induced field automorphism
$\sigma\in\operatorname{Aut}_K(\mcF)$, where
$K=\mathbb Q(\widehat\mcT,z)$, also has infinite order.
\end{lemma}

\begin{proof}
We denote
$\sigma=\widehat\mu_1\widehat\mu_2$. Specialize
$z=1$ and then evaluate at $(x,y)=(s,1)$, where $s$ is an
indeterminate. Define $(x_n,y_n)\in\mathbb Q(s)^2$ recursively by
$
(x_n,y_n,1)=\sigma^n(s,1,1).
$
The mutation formulas give
\begin{align}
y_{n+1}=\frac{x_n^2+k_2x_n+1}{y_n},\
x_{n+1}=\frac{y_{n+1}^2+k_1y_{n+1}+1}{x_n}.
\label{eq:I4-degree-recurrence}
\end{align}

For a nonzero rational function $r(s)=p(s)/q(s)$, with coprime
$p,q\in\mathbb Q[s]$, denote $\deg_s r=\deg p-\deg q$. We claim that
\begin{align}
\deg_s x_n=2n+1,\ \deg_s y_n=2n
\label{eq:I4-degree-growth}
\end{align}
for every $n\geq0$. This is clear for $n=0$. Suppose that it holds for
some $n$. Since $\deg_s x_n=2n+1>0$, the first
equality in
\eqref{eq:I4-degree-recurrence} therefore gives
$\deg_s y_{n+1}=2(2n+1)-2n=2n+2$.
Similarly, by $\deg_s y_{n+1}=2n+2>0$, the
second equality in \eqref{eq:I4-degree-recurrence} yields
$\deg_s x_{n+1}=2(2n+2)-(2n+1)=2n+3$.
This proves \eqref{eq:I4-degree-growth} by induction.

If $\sigma^m$ is the identity as a birational transformation for some
$m\geq1$, then the above specialization gives $x_m=s$, and therefore
$\deg_s x_m=1$. On the other hand,
\eqref{eq:I4-degree-growth} gives $\deg_s x_m=2m+1>1$, which is a
contradiction. Thus, $\sigma$ has infinite order in $\widehat\Gamma$.
Since $\rho$ is injective, the field automorphism induced by $\sigma$
also has infinite order.
\end{proof}

\begin{lemma}
\label{lem:fixed-field-one-variable}
Let $E/K$ be a regular function field of transcendence degree one. If
$\sigma\in\operatorname{Aut}_K(E)$ has infinite order, then
$
E^{\langle\sigma\rangle}=K.
$
\end{lemma}

\begin{proof}
Suppose that $f\in E^{\langle\sigma\rangle}\setminus K$. Since $K$ is
algebraically closed in $E$, the element $f$ is transcendental over
$K$. Thus, $\operatorname{trdeg}_K K(f)=1=\operatorname{trdeg}_K E$,
which implies that $E/K(f)$ is algebraic. Since $E/K$ is finitely generated, this
algebraic extension is finite, and hence $[E:K(f)]<\infty$. But
$\sigma$ fixes $K(f)$ pointwise, so we have 
$
\langle\sigma\rangle\leq\operatorname{Aut}_{K(f)}(E).
$
Note that the standard estimate
\begin{align*}
|\operatorname{Aut}_{K(f)}(E)|\leq [E:K(f)]<\infty
\end{align*}
shows that the group $\operatorname{Aut}_{K(f)}(E)$ is finite, which
contradicts the fact that $\sigma$ has infinite order. Hence, we obtain that $E^{\langle\sigma\rangle}=K$.
\end{proof}

\subsection{The invariant field and its Laurent part} We prove $\widehat\mcT$ in
\eqref{eq:I4-basic-invariant} is a basic invariant.
\begin{theorem}
\label{thm:I4-generalized-invariant-ring}
Let
$\mathcal L=\mathbb Q[x^{\pm1},y^{\pm1},z^{\pm1}]$ and
$\widehat\Gamma=\langle\widehat\mu_1,\widehat\mu_2,
\widehat\mu_3\rangle$ be given by
\eqref{eq:I4-generalized-mutations}. Then
\begin{align*}
\mcF^{\widehat\Gamma}=\mathbb Q(\widehat\mcT),\
\mcF^{\widehat\Gamma}\cap\mathcal L
=\mathbb Q[\widehat\mcT].
\end{align*}
\end{theorem}

\begin{proof}
Write $t=\widehat\mcT$ and $K=\mathbb Q(t,z)$. The mutations
$\widehat\mu_1$ and $\widehat\mu_2$ fix $K$ pointwise. By
\Cref{prop:I4-regular-extension}, \Cref{lem:I4-infinite-order} and
Lemma~\ref{lem:fixed-field-one-variable}, we have
$
\mcF^{\langle\widehat\mu_1\widehat\mu_2\rangle}=K.
$

Let $f\in\mcF^{\widehat\Gamma}$. Then $f=r(t,z)$ for some
$r(t,Z)\in\mathbb Q(t)(Z)$. Let $z'=\widehat\mu_3(z)$. By
\eqref{eq:I4-basic-invariant}, we have
\begin{align*}
z'=\frac{x^2+k_3xy+y^2}{z}
=txy-z-k_1y-k_2x.
\end{align*}
Since $\widehat\mu_1$ fixes $t,y,z$, it follows that
$
\widehat\mu_1(z')-z'
=(ty-k_2)\bigl(\widehat\mu_1(x)-x\bigr)\neq0.
$
Indeed, both factors are nonzero because $x,y,z$ are algebraically
independent over $\mathbb Q$. Hence, $z'\notin K$. Since $\mcF/K$ is
regular by \Cref{prop:I4-regular-extension}, the element $z'$ is
transcendental over $K$.

The equality $\widehat\mu_3(f)=f$ gives $r(t,z')=r(t,z)$. If
$r(t,Z)$ is not constant in $Z$, write $r=G/H$ with coprime
$G,H\in\mathbb Q(t)[Z]$. Then, $z'$ is a root of the nonzero polynomial
\begin{align*}
G(Z)H(z)-G(z)H(Z)\in K[Z],
\end{align*}
contradicting the transcendence of $z'$ over $K$. Therefore, $r$ is
independent of $z$, and we obtain
\begin{align*}
\mcF^{\widehat\Gamma}=\mathbb Q(\widehat\mcT).
\end{align*}

It remains to intersect with $\mathcal L$. Let
$h(\widehat\mcT)=P(\widehat\mcT)/Q(\widehat\mcT)\in\mathcal L$, where
$P,Q\in\mathbb Q[T]$ are coprime. The B\'ezout's identity for $P$ and $Q$ shows that 
$Q(\widehat\mcT)$ is a unit of $\mathcal L$. Suppose that $Q$ is
nonconstant, and choose a root $\alpha\in\overline{\mathbb Q}$ of
$Q$. For any $a\in\overline{\mathbb Q}^{\times}$, we have
\begin{align*}
\widehat\mcT(a,1,1)
=a+(k_2+k_3)+\frac{2+k_1}{a}.
\end{align*}
Choose a root $a$ of
$
a^2+(k_2+k_3-\alpha)a+(2+k_1)=0.
$
Since $k_1\in\mbN$, its constant term is nonzero, and hence $a\neq0$.
It follows that $\widehat\mcT(a,1,1)=\alpha$, which implies that
$Q(\widehat\mcT)(a,1,1)=0$. This contradicts the fact that
$Q(\widehat\mcT)$ is a unit of $\mathcal L$. Thus, $Q$ is constant,
and consequently
\begin{align*}
\mathbb Q(\widehat\mcT)\cap\mathcal L
=\mathbb Q[\widehat\mcT].
\end{align*}
\end{proof}

\section{On a conjecture of Chen-Li}
\label{sec:proof-chen-li-conjecture}
In this section, as an application of \Cref{thm:I4-generalized-invariant-ring}, we prove a conjecture proposed by Chen-Li in \cite{CL25}.
Let $\mcF=\mathbb Q(x,y,z)$ and
$\mathcal L=\mathbb Q[x^{\pm1},y^{\pm1},z^{\pm1}]$. For the irreducible sign-equivalent exchange matrix $B_i$ in \eqref{sign-equivalence class}, we denote 
$\Gamma_{B_i}=\langle\mu_1,\mu_2,\mu_3\rangle$. The three Markov-type equations in
Table~\ref{tab: classical} correspond to the
following basic Laurent mutation invariants:
\begin{equation} \label{eq:rank-3-laurent-conjecture}
\begin{aligned}
\mcT_1=\frac{x^2+y^2+z^2}{xyz},\
\mcT_2=\frac{x^2+y^4+z^4+2xy^2+2xz^2}{xy^2z^2},\
\mcT_3=\frac{x^4+y^2+z^2+2yz}{x^2yz}.
\end{aligned}
\end{equation}
Motivated by \cite[Conjecture~6.9]{CL25} about $\mcT_1,\mcT_2$ and the Laurent mutation
invariant $\mcT_3$ found in \cite{Kau24}, we formulate the
following conjecture.

\begin{conjecture}[{cf. \cite[Conjecture 6.9]{CL25}}]\label{conj: rank 3}
Up to a permutation of the variables, let $\mcT_i$ denote the one among
$\mcT_1,\mcT_2,\mcT_3$ corresponding to the exchange matrix $B_i$ in \eqref{sign-equivalence class}.
Then, every Laurent mutation invariant is
a polynomial in $\mcT_i$. That is to say,
\setcounter{equation}{6}
\begin{equation*}
\mcF^{\Gamma_{B_i}}\cap\mathcal L=\mathbb Q[\mcT_i].
\end{equation*}
\end{conjecture}

\begin{lemma}
\label{lem:finite-equivariant-descent}
Let $F/E$ be a finite field extension, and let
$G\leq\operatorname{Aut}(F)$ be a group such that
$
g(E)=E
$ for every $g\in G$.
Thus, the action of $G$ on $F$ restricts to an action on the subfield
$E$. Denote the corresponding fixed fields respectively by
$
F^G,
E^G.
$
Then, every $f\in F^G$ is algebraic over $E^G$.
\end{lemma}

\begin{proof}
Since $F/E$ is finite, the element $f$ is algebraic over $E$. Let
$p(\lambda)\in E[\lambda]$ be its monic minimal polynomial. Since
$g(E)=E$, every $g\in G$ acts coefficientwise on $E[\lambda]$. Write
\begin{align*}
p(\lambda)=\lambda^d+a_{d-1}\lambda^{d-1}+\cdots+a_0
\end{align*}
and we have
$
g(p)(\lambda)=\lambda^d+g(a_{d-1})\lambda^{d-1}+\cdots+g(a_0)
\in E[\lambda].
$
Because $f\in F^G$, we have $g^{-1}(f)=f$. Therefore,
\begin{align*}
g(p)(g(b))
&=\sum_{j=0}^d g(a_j)g(b)^j
=g(\sum_{j=0}^d a_jb^j)
=g(p(b))
\end{align*}
for every $b\in F$. Taking $b=g^{-1}(f)$ gives
\begin{align*}
g(p)(f)=g(p(g^{-1}(f)))=g(p(f))=0,
\end{align*}
where the second equality follows from $g^{-1}(f)=f$.
Moreover, $g(p)(\lambda)$ is monic and irreducible over $E$, since
$g$ is an automorphism of $E$. Thus, $g(p)(\lambda)$ is also the monic minimal
polynomial of $f$ over $E$. By its uniqueness, $g(p)=p$ for every
$g\in G$. Hence, each coefficient $a_j$ is fixed by every element of
$G$, which implies that $a_j\in E^G$. Consequently,
$p(\lambda)\in E^G[\lambda]$, which proves
that $f$ is algebraic over $E^G$.
\end{proof}

\begin{proposition}
\label{prop:classical-regularity-after-substitution}
For $i=2,3$, let $K_i=\mathbb Q(\mcT_i,z)$. Then
$\mcF/K_i$ is a regular function field of transcendence degree one.
\end{proposition}

\begin{proof}
Let $t=\mcT_i$. For $i=2$, take $(U,W)=(x,y)$, while for $i=3$, take
$(U,W)=(y,x)$. In the two respective cases, $U$ satisfies the
quadratic equation $f_i(U)=0$ listed in
Table~\ref{tab:quadratic-models-after-substitution}.
\begin{table}[htbp]
\centering
\begin{tabular}{c l}
\toprule
$i$ & $f_i(U)$\\
\midrule
$2$ & $U^2+(2W^2+2z^2-tW^2z^2)U+W^4+z^4$\\
$3$ & $U^2+(2z-tzW^2)U+W^4+z^2$\\
\bottomrule
\end{tabular}
\par\vspace{0.4em}
\caption{Quadratic equations associated with $\mcT_2$ and $\mcT_3$}
\label{tab:quadratic-models-after-substitution}
\end{table}
\par
Hence, $\mcF=\mbQ(U,W,z)$ is algebraic over $\mathbb Q(t,W,z)$,
which implies that $t,W,z$ are algebraically independent and
$\operatorname{trdeg}_{K_i}\mcF=1$.
Let $\overline{K_i}$ be an algebraic closure of $K_i$. The
discriminants of the above quadratics are listed in
Table~\ref{tab:quadratic-discriminants-after-substitution}.
\begin{table}[H]
\centering
\begin{tabular}{c l}
\toprule
$i$ & $\Delta_i(W)$\\
\midrule
$2$ & $W^2z^2\bigl((t^2z^2-4t)W^2+8-4tz^2\bigr)$\\
$3$ & $W^2\bigl((t^2z^2-4)W^2-4tz^2\bigr)$\\
\bottomrule
\end{tabular}
\par\vspace{0.4em}
\caption{Discriminants of the quadratic equations in
Table~\ref{tab:quadratic-models-after-substitution}}
\label{tab:quadratic-discriminants-after-substitution}
\end{table}
\par
By an argument similar to that used in the proof of \Cref{prop:I4-regular-extension}, one can show that $\Delta_i(W)$ is not a square in $\overline{K_i}(W)$ and $f_i(X)$ is irreducible over
$\overline{K_i}(W)$ (and hence over $K_i(W)$). Moreover, we have 
\begin{align*}
\mcF\otimes_{K_i}\overline{K_i}
\cong\overline{K_i}(W)[X]/(f_i(X)).
\end{align*}
The right hand side is a field, so Lemma~\ref{lem:regularity-criterion}
proves that $\mcF/K_i$ is regular for every $i\in\{2,3\}$.
\end{proof}

\begin{theorem}
\label{thm:classical-from-I4}
For each $i\in\{1,2,3\}$, we have
\begin{align*}
\mcF^{\Gamma_{B_i}}=\mathbb Q(\mcT_i),\
\mcF^{\Gamma_{B_i}}\cap\mathcal L=\mathbb Q[\mcT_i].
\end{align*}
In particular,
\Cref{conj: rank 3} holds.
\end{theorem}

\begin{proof}
For $i=1$, take $(k_1,k_2,k_3)=(0,0,0)$ in
\eqref{eq:I4-basic-invariant}. Then $\widehat\mcT_{0,0,0}=\mcT_1$, and the
generalized mutations in \eqref{eq:I4-generalized-mutations} reduce to
the classical mutations associated with $B_1$. Therefore, the assertion
follows directly from \Cref{thm:I4-generalized-invariant-ring}.

For $i=2$, let
$
S_2(x,y,z)=(x,y^2,z^2).
$
Take $(k_1,k_2,k_3)=(0,2,2)$. Then
\begin{align*}
\widehat\mcT_{0,2,2}(S_2(x,y,z))
&=\frac{x^2+y^4+z^4+2xz^2+2xy^2}{xy^2z^2}
=\mcT_2.
\end{align*}
By a direct calculation, the corresponding generalized and classical
mutations are compatible through $S_2$, namely
$\widehat\mu_j\circ S_2=S_2\circ\mu_j$ for $j=1,2,3$.
Thus, the subfield
$E_2=\mathbb Q(x,y^2,z^2)$ is preserved by $\Gamma_{B_2}$, and
\Cref{thm:I4-generalized-invariant-ring} gives
\begin{align*}
E_2^{\Gamma_{B_2}}=\mathbb Q(\mcT_2).
\end{align*}
Since $\mcF/E_2$ is finite, every element of
$\mcF^{\Gamma_{B_2}}$ is algebraic over $\mathbb Q(\mcT_2)$ by
Lemma~\ref{lem:finite-equivariant-descent}.

For $i=3$, define
$
S_3(x,y,z)=(x^2,y,z).
$
Taking $(k_1,k_2,k_3)=(2,0,0)$ gives
\begin{align*}
\widehat\mcT_{2,0,0}(S_3(x,y,z))
&=\frac{x^4+y^2+z^2+2yz}{x^2yz}
=\mcT_3.
\end{align*}
Similarly, we have 
$\widehat\mu_j\circ S_3=S_3\circ\mu_j$ for $j=1,2,3$; see also \Cref{contract}. Thus, the subfield $E_3=\mathbb Q(x^2,y,z)$ is preserved by $\Gamma_{B_3}$, and
\Cref{thm:I4-generalized-invariant-ring} gives
\begin{align*}
E_3^{\Gamma_{B_3}}=\mathbb Q(\mcT_3).
\end{align*}
Again, $\mcF/E_3$ is finite, and
Lemma~\ref{lem:finite-equivariant-descent} shows that every element of
$\mcF^{\Gamma_{B_3}}$ is algebraic over $\mathbb Q(\mcT_3)$.

Fix $i\in\{2,3\}$ and let
$f\in\mcF^{\Gamma_{B_i}}$. By the preceding application of
Lemma~\ref{lem:finite-equivariant-descent}, the element $f$ is
algebraic over $\mathbb Q(\mcT_i)$. Let
$K_i=\mathbb Q(\mcT_i,z)$. By
\Cref{prop:classical-regularity-after-substitution}, the extension
$\mcF/K_i$ is regular. Since $\mathbb Q(\mcT_i)\subseteq K_i$, the
element $f$ is also algebraic over $K_i$. Regularity implies that
$K_i$ is algebraically closed in $\mcF$, and hence $f\in K_i$.

It remains to show that $f$ is independent of $z$. Let
$L_i=\mathbb Q(\mcT_i)$ and then $K_i=L_i(z)$. Note that $z$ is
transcendental over $L_i$. Suppose, to the contrary, that $f\notin
L_i$. Since $f\in K_i=L_i(z)$, the element $f$ is a nonconstant
rational function of $z$. Writing
$f=P(z)/Q(z)$ with $P,Q\in L_i[Z]$, we see that $z$ is a root of
$
P(Z)-fQ(Z)
$
over $L_i(f)$. This polynomial is nonzero because
$f\notin L_i$. Thus, $z$ is algebraic over $L_i(f)$. On the other
hand, we already know from Lemma~\ref{lem:finite-equivariant-descent}
that $f$ is algebraic over $L_i$. Hence, $L_i(f)/L_i$ is algebraic,
and the transitivity of algebraic extensions would imply that $z$ is
algebraic over $L_i$. This contradicts the transcendence of $z$ over
$L_i$, which implies that  
$f\in L_i=\mathbb Q(\mcT_i)$. Since $f$ is an arbitrary element of
$\mcF^{\Gamma_{B_i}}$, we obtain
\begin{align*}
\mcF^{\Gamma_{B_i}}=\mathbb Q(\mcT_i),
\quad i=2,3.
\end{align*}
Moreover, by arguments analogous to those in \Cref{thm:I4-generalized-invariant-ring}, we have  
\begin{align*}
\mathbb Q(\mcT_i)\cap\mathcal L=\mathbb Q[\mcT_i].
\end{align*}
\end{proof}

\begin{remark}
In particular, \Cref{thm:classical-from-I4} also answers the
question posed in \cite[Question~2.29]{CL24}, since rank $2$ cluster
algebras of affine type may be viewed as degenerations of the rank $3$
cases considered above, obtained by specializing one variable to $1$.
\end{remark}

%=======================================
\section{Further discussion}\label{sec: further-questions}
For generalized cluster algebras and generalized Markov equations, we conclude with an open question and two ordered uniqueness conjectures.

\begin{question}\label{question: open}
	For the following class of Diophantine equations
	\begin{align*}
x^2+y^2+z^2+k_1yz+k_2xz+k_3xy=kxyz,
	\end{align*} where $k_1,k_2,k_3\in \mbN$, and $k\in \mbN_+$, classify all the choices of $k$ such that this equation has positive integer solutions. Moreover, find out all the $\widehat{\Gamma}$-orbits of the solutions.
\end{question}
\begin{remark} There are several observations and computational experiments as follows.

\begin{enumerate}
	\item 
In fact, by \cite{GM23}, when $k=3+k_1+k_2+k_3$, all the positive integer solutions to 
\begin{align}
x^2+y^2+z^2+k_1yz+k_2xz+k_3xy=(3+k_1+k_2+k_3)xyz\label{equ: gme}
\end{align}
lie in the unique orbit $\widehat{\Gamma}[(1,1,1)]$. However, for general $k$, it is still a difficult but interesting problem since we are even not clear about the existence of the solutions. 
	\item We consider the following Diophantine equation which does not belong to the above class
	\begin{align*}x^2+y^2+z^2+2yz+2k_2xz+2k_3xy=(2+k_2+k_3)xyz.
	\end{align*} A direct calculation shows that $(4,2,2)$ is a positive integer solution. When $k_2=2,k_3=0$, there is another solution $(2,3,1)$. Moreover, different choices of $k_2$ and $k_3$ may lead to different orbits of solutions. It is natural to study how $k_2$ and $k_3$ control the orbits.
	\end{enumerate}
\end{remark}
Finally, motivated by the classical uniqueness conjecture for the Markov equation, the following ordered uniqueness conjecture has been proposed for generalized Markov equations \eqref{equ: gme} in the symmetric case $k_1=k_2=k_3$. However, without this assumption, the analogous statements may fail as illustrated below.
\begin{conjecture}[{\cite{GM24} \& \cite{Fro13}}] \label{conj: CJ} Suppose that $k_1=k_2=k_3$. If $(A,B,C)$ and $(A,B^{\prime},C^{\prime})$ are two positive integer solutions to the generalized Markov equation \eqref{equ: gme} with $A\geq B\geq C$ and $A\geq B^{\prime}\geq C^{\prime}$, then $B=B^{\prime}$ and $C=C^{\prime}$.
\end{conjecture}

\begin{conjecture}\label{conj: uniqueness-CH}  For $k=2,4,5$ in Markov-type equation \eqref{equ: CH}, the following statements hold by symmetry in the variables $y$ and $z$. 

\begin{enumerate}
	\item If $(A,B,C)$ and $(A,B^{\prime},C^{\prime})$ are two positive integer solutions with $A\geq B\geq C$ and $A\geq B^{\prime}\geq C^{\prime}$, then $B=B^{\prime}$ and $C=C^{\prime}$.
	\item If $(A,B,C)$ and $(A^{\prime},B,C^{\prime})$ are two positive integer solutions with $B\geq A\geq C$ and $B\geq A^{\prime}\geq C^{\prime}$, then $A=A^{\prime}$ and $C=C^{\prime}$.
	\item If $(A,B,C)$ and $(A^{\prime},B,C^{\prime})$ are two positive integer solutions with $B\geq C\geq A$ and $B\geq C^{\prime}\geq A^{\prime}$, then $A=A^{\prime}$ and $C=C^{\prime}$.
\end{enumerate}
\end{conjecture}

The restriction on $k$ is necessary: when $k=3$, $(75,13,2)$ and $(75,29,1)$ are distinct positive integer solutions to \eqref{equ: CH} satisfying $75>13>2$ and $75>29>1$, respectively. By \Cref{prop: scaling correspondence}, multiplying these triples by $3$ also gives a counterexample for $k=1$. To some degree, this also reflects the underlying distinction in orbit structure: the cases $k=1,3$ have multiple orbits, whereas each of $k=2,4,5$ has a unique orbit.
%=======================================
\clearpage
\section{Appendix: orbits of solutions}\label{sec:orbit-appendix}
In this section, the appendix illustrates the orbits of solutions that occur in  \Cref{thm: main}. These diagrams display the initial branching of each orbit and indicate how further solutions are generated by successive generalized cluster mutations. Since the variables $y$ and $z$ are symmetric, we display the orbits only up to the interchange of $y$ and $z$.

\begingroup
\setlength{\intextsep}{2pt}
\setlength{\abovedisplayskip}{0pt}
\setlength{\belowdisplayskip}{0pt}
\setlength{\abovecaptionskip}{2pt}
\setlength{\belowcaptionskip}{0pt}
\setlength{\captionindent}{0pt}
\makeatletter
\def\@captionfont{\small}
\makeatother
\noindent\begin{minipage}[t]{0.49\textwidth}
\centering
\begin{align*}
\scalebox{0.64}{
\begin{xy}(0,0)*+{(9,6,3)}="0",(20,20)*+{(9,6,3)}="1",(20,0)*+{(9,15,3)}="1'",(20,-20)*+{(9,6,39)}="1''",(45,50)*+{(9,15,3)}="20",(45,30)*+{(9,6,39)}="21",(45,10)*+{(36,15,3)}="22",(45,-10)*+{(9,15,102)}="23",(45,-30)*+{(225,6,39)}="24",(45,-50)*+{(9,267,39)}="25",(80,55)*+{(36,15,3)\cdots}="40",(80,45)*+{(9,15,102)\cdots}="41", (80,35)*+{(225,6,39)\cdots}="42", (80,25)*+{(9,267,39)\cdots}="43", (80,15)*+{(36,87,3)\cdots}="44", (80,5)*+{(36,15,507)\cdots}="45", (80,-5)*+{(1521,15,102)\cdots}="46", (80,-15)*+{(9,699,102)\cdots}="47", (80,-25)*+{(225,8691,39)\cdots}="48", (80,-35)*+{(225,6,1299)\cdots}="49", (80,-45)*+{(10404,267,39)\cdots}="410", (80,-55)*+{(9,267,1830)\cdots}="411", \ar@{-}^{\mu_1}"0";"1"\ar@{-}^{\mu_2}"0";"1'"\ar@{-}_{\mu_3}"0";"1''"\ar@{-}^{\mu_2}"1";"20"\ar@{-}_{\mu_3}"1";"21"\ar@{-}^{\mu_1}"1'";"22"\ar@{-}_{\mu_3}"1'";"23"\ar@{-}^{\mu_1}"1''";"24"\ar@{-}_{\mu_2}"1''";"25"\ar@{-}^{\mu_1}"20";"40"\ar@{-}_{\mu_3}"20";"41"\ar@{-}^{\mu_1}"21";"42"\ar@{-}_{\mu_2}"21";"43"\ar@{-}^{\mu_2}"22";"44"\ar@{-}_{\mu_3}"22";"45"\ar@{-}^{\mu_1}"23";"46"\ar@{-}_{\mu_2}"23";"47"\ar@{-}^{\mu_2}"24";"48"\ar@{-}_{\mu_3}"24";"49"\ar@{-}^{\mu_1}"25";"410"\ar@{-}_{\mu_3}"25";"411"
\end{xy}}\notag
\end{align*}

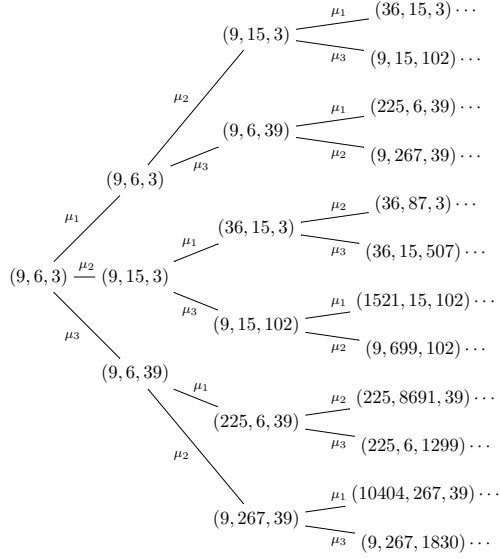
\captionof{figure}{The orbit $\widehat{\Gamma}[(9,6,3)]$ for $k=1$}
\label{1 orbit for $k=1$}
\end{minipage}\hfill%
\begin{minipage}[t]{0.49\textwidth}
\centering
\begin{align*}
\scalebox{0.64}{
\begin{xy}(0,0)*+{(8,4,4)}="0",(20,20)*+{(8,4,4)}="1",(20,0)*+{(8,20,4)}="1'",(20,-20)*+{(8,4,20)}="1''",(45,50)*+{(8,20,4)}="20",(45,30)*+{(8,4,20)}="21",(45,10)*+{(72,20,4)}="22",(45,-10)*+{(8,20,116)}="23",(45,-30)*+{(72,4,20)}="24",(45,-50)*+{(8,116,20)}="25",(80,55)*+{(72,20,4)\cdots}="40",(80,45)*+{(8,20,116)\cdots}="41", (80,35)*+{(72,4,20)\cdots}="42", (80,25)*+{(8,140,20)\cdots}="43", (80,15)*+{(72,260,4)\cdots}="44", (80,5)*+{(72,20,1396)\cdots}="45", (80,-5)*+{(2312,20,116)\cdots}="46", (80,-15)*+{(8,676,116)\cdots}="47", (80,-25)*+{(72,1396,20)\cdots}="48", (80,-35)*+{(72,4,260)\cdots}="49", (80,-45)*+{(2312,116,20)\cdots}="410", (80,-55)*+{(8,116,676)\cdots}="411", \ar@{-}^{\mu_1}"0";"1"\ar@{-}^{\mu_2}"0";"1'"\ar@{-}_{\mu_3}"0";"1''"\ar@{-}^{\mu_2}"1";"20"\ar@{-}_{\mu_3}"1";"21"\ar@{-}^{\mu_1}"1'";"22"\ar@{-}_{\mu_3}"1'";"23"\ar@{-}^{\mu_1}"1''";"24"\ar@{-}_{\mu_2}"1''";"25"\ar@{-}^{\mu_1}"20";"40"\ar@{-}_{\mu_3}"20";"41"\ar@{-}^{\mu_1}"21";"42"\ar@{-}_{\mu_2}"21";"43"\ar@{-}^{\mu_2}"22";"44"\ar@{-}_{\mu_3}"22";"45"\ar@{-}^{\mu_1}"23";"46"\ar@{-}_{\mu_2}"23";"47"\ar@{-}^{\mu_2}"24";"48"\ar@{-}_{\mu_3}"24";"49"\ar@{-}^{\mu_1}"25";"410"\ar@{-}_{\mu_3}"25";"411"
\end{xy}}\notag
\end{align*}

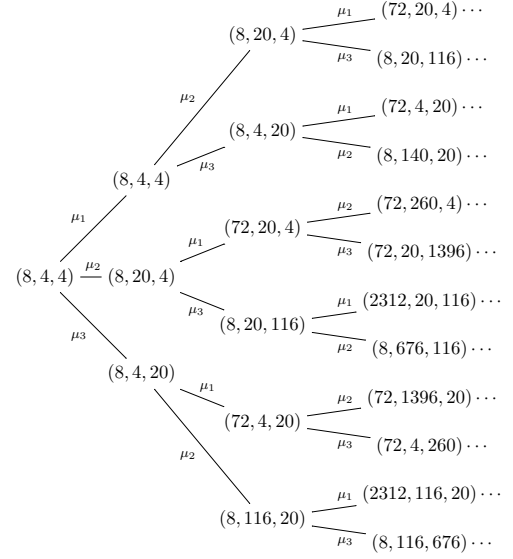
\captionof{figure}{The orbit $\widehat{\Gamma}[(8,4,4)]$ for $k=1$}
\label{2 orbit for $k=1$}
\end{minipage}
\par\vspace{16pt}

\noindent\begin{minipage}[t]{0.49\textwidth}
\centering
\begin{align*}
\scalebox{0.64}{
\begin{xy}(0,0)*+{(5,5,5)}="0",(20,20)*+{(20,5,5)}="1",(20,0)*+{(5,10,5)}="1'",(20,-20)*+{(5,5,10)}="1''",(45,50)*+{(20,85,5)}="20",(45,30)*+{(20,5,85)}="21",(45,10)*+{(45,10,5)}="22",(45,-10)*+{(5,10,25)}="23",(45,-30)*+{(45,5,10)}="24",(45,-50)*+{(5,25,10)}="25",(80,55)*+{(405,85,5)\cdots}="40",(80,45)*+{(20,85,1525)\cdots}="41", (80,35)*+{(405,5,85)\cdots}="42", (80,25)*+{(20,1525,85)\cdots}="43", (80,15)*+{(45,205,5)\cdots}="44", (80,5)*+{(45,10,425)\cdots}="45", (80,-5)*+{(245,10,25)\cdots}="46", (80,-15)*+{(5,65,25)\cdots}="47", (80,-25)*+{(45,425,10)\cdots}="48", (80,-35)*+{(45,5,205)\cdots}="49", (80,-45)*+{(245,25,10)\cdots}="410", (80,-55)*+{(5,25,65)\cdots}="411", \ar@{-}^{\mu_1}"0";"1"\ar@{-}^{\mu_2}"0";"1'"\ar@{-}_{\mu_3}"0";"1''"\ar@{-}^{\mu_2}"1";"20"\ar@{-}_{\mu_3}"1";"21"\ar@{-}^{\mu_1}"1'";"22"\ar@{-}_{\mu_3}"1'";"23"\ar@{-}^{\mu_1}"1''";"24"\ar@{-}_{\mu_2}"1''";"25"\ar@{-}^{\mu_1}"20";"40"\ar@{-}_{\mu_3}"20";"41"\ar@{-}^{\mu_1}"21";"42"\ar@{-}_{\mu_2}"21";"43"\ar@{-}^{\mu_2}"22";"44"\ar@{-}_{\mu_3}"22";"45"\ar@{-}^{\mu_1}"23";"46"\ar@{-}_{\mu_2}"23";"47"\ar@{-}^{\mu_2}"24";"48"\ar@{-}_{\mu_3}"24";"49"\ar@{-}^{\mu_1}"25";"410"\ar@{-}_{\mu_3}"25";"411"
\end{xy}}\notag
\end{align*}

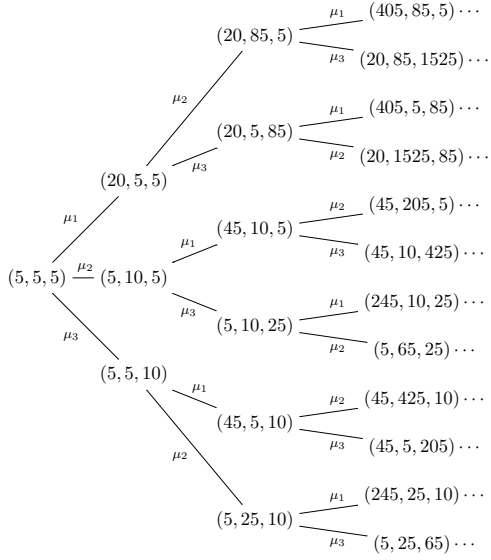
\captionof{figure}{The orbit $\widehat{\Gamma}[(5,5,5)]$ for $k=1$}
\label{3 orbit for $k=1$}
\end{minipage}\hfill%
\begin{minipage}[t]{0.49\textwidth}
\centering
\begin{align*}
\scalebox{0.64}{
\begin{xy}(0,0)*+{(4,2,2)}="0",(20,20)*+{(4,2,2)}="1",(20,0)*+{(4,10,2)}="1'",(20,-20)*+{(4,2,10)}="1''",(45,50)*+{(4,10,2)}="20",(45,30)*+{(4,2,10)}="21",(45,10)*+{(36,10,2)}="22",(45,-10)*+{(4,10,58)}="23",(45,-30)*+{(36,2,10)}="24",(45,-50)*+{(4,58,10)}="25",(80,55)*+{(36,10,2)\cdots}="40",(80,45)*+{(4,10,58)\cdots}="41", (80,35)*+{(36,2,10)\cdots}="42", (80,25)*+{(4,58,10)\cdots}="43", (80,15)*+{(36,130,2)\cdots}="44", (80,5)*+{(36,10,698)\cdots}="45", (80,-5)*+{(1156,10,58)\cdots}="46", (80,-15)*+{(4,338,58)\cdots}="47", (80,-25)*+{(36,698,10)\cdots}="48", (80,-35)*+{(36,2,130)\cdots}="49", (80,-45)*+{(1156,58,10)\cdots}="410", (80,-55)*+{(4,58,338)\cdots}="411", \ar@{-}^{\mu_1}"0";"1"\ar@{-}^{\mu_2}"0";"1'"\ar@{-}_{\mu_3}"0";"1''"\ar@{-}^{\mu_2}"1";"20"\ar@{-}_{\mu_3}"1";"21"\ar@{-}^{\mu_1}"1'";"22"\ar@{-}_{\mu_3}"1'";"23"\ar@{-}^{\mu_1}"1''";"24"\ar@{-}_{\mu_2}"1''";"25"\ar@{-}^{\mu_1}"20";"40"\ar@{-}_{\mu_3}"20";"41"\ar@{-}^{\mu_1}"21";"42"\ar@{-}_{\mu_2}"21";"43"\ar@{-}^{\mu_2}"22";"44"\ar@{-}_{\mu_3}"22";"45"\ar@{-}^{\mu_1}"23";"46"\ar@{-}_{\mu_2}"23";"47"\ar@{-}^{\mu_2}"24";"48"\ar@{-}_{\mu_3}"24";"49"\ar@{-}^{\mu_1}"25";"410"\ar@{-}_{\mu_3}"25";"411"
\end{xy}}\notag
\end{align*}

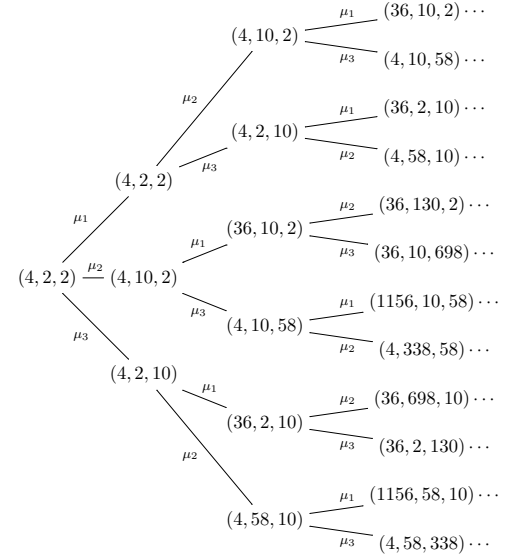
\captionof{figure}{The orbit $\widehat{\Gamma}[(4,2,2)]$ for $k=2$}
\label{1 orbit for $k=2$}
\end{minipage}

\clearpage
\noindent\begin{minipage}[t]{0.49\textwidth}
\centering
\begin{align*}
\scalebox{0.64}{
\begin{xy}(0,0)*+{(3,2,1)}="0",(20,20)*+{(3,2,1)}="1",(20,0)*+{(3,5,1)}="1'",(20,-20)*+{(3,2,13)}="1''",(45,50)*+{(3,5,1)}="20",(45,30)*+{(3,2,13)}="21",(45,10)*+{(12,5,1)}="22",(45,-10)*+{(3,5,34)}="23",(45,-30)*+{(75,2,13)}="24",(45,-50)*+{(3,89,13)}="25",(80,55)*+{(12,5,1)\cdots}="40",(80,45)*+{(3,5,34)\cdots}="41", (80,35)*+{(75,2,13)\cdots}="42", (80,25)*+{(3,89,13)\cdots}="43", (80,15)*+{(12,29,1)\cdots}="44", (80,5)*+{(12,5,169)\cdots}="45", (80,-5)*+{(507,5,34)\cdots}="46", (80,-15)*+{(3,233,34)\cdots}="47", (80,-25)*+{(75,2897,13)\cdots}="48", (80,-35)*+{(75,2,433)\cdots}="49", (80,-45)*+{(3468,89,13)\cdots}="410", (80,-55)*+{(3,89,610)\cdots}="411", \ar@{-}^{\mu_1}"0";"1"\ar@{-}^{\mu_2}"0";"1'"\ar@{-}_{\mu_3}"0";"1''"\ar@{-}^{\mu_2}"1";"20"\ar@{-}_{\mu_3}"1";"21"\ar@{-}^{\mu_1}"1'";"22"\ar@{-}_{\mu_3}"1'";"23"\ar@{-}^{\mu_1}"1''";"24"\ar@{-}_{\mu_2}"1''";"25"\ar@{-}^{\mu_1}"20";"40"\ar@{-}_{\mu_3}"20";"41"\ar@{-}^{\mu_1}"21";"42"\ar@{-}_{\mu_2}"21";"43"\ar@{-}^{\mu_2}"22";"44"\ar@{-}_{\mu_3}"22";"45"\ar@{-}^{\mu_1}"23";"46"\ar@{-}_{\mu_2}"23";"47"\ar@{-}^{\mu_2}"24";"48"\ar@{-}_{\mu_3}"24";"49"\ar@{-}^{\mu_1}"25";"410"\ar@{-}_{\mu_3}"25";"411"
\end{xy}}\notag
\end{align*}

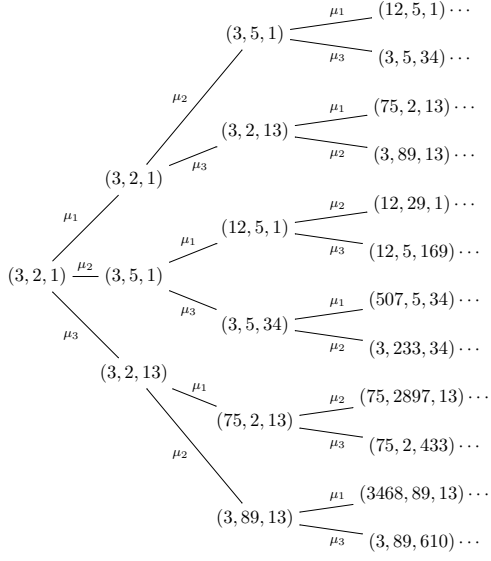
\captionof{figure}{The orbit $\widehat{\Gamma}[(3,2,1)]$ for $k=3$}
\label{1 orbit for $k=3$}
\end{minipage}\hfill%
\begin{minipage}[t]{0.49\textwidth}
\centering
\begin{align*}
\scalebox{0.64}{
\begin{xy}(0,0)*+{(2,1,1)}="0",(20,20)*+{(2,1,1)}="1",(20,0)*+{(2,5,1)}="1'",(20,-20)*+{(2,1,5)}="1''",(45,50)*+{(2,5,1)}="20",(45,30)*+{(2,1,5)}="21",(45,10)*+{(18,5,1)}="22",(45,-10)*+{(2,5,29)}="23",(45,-30)*+{(18,1,5)}="24",(45,-50)*+{(2,29,5)}="25",(80,55)*+{(18,5,1)\cdots}="40",(80,45)*+{(2,5,29)\cdots}="41", (80,35)*+{(18,1,5)\cdots}="42", (80,25)*+{(2,29,5)\cdots}="43", (80,15)*+{(18,65,1)\cdots}="44", (80,5)*+{(18,5,349)\cdots}="45", (80,-5)*+{(578,5,29)\cdots}="46", (80,-15)*+{(2,169,29)\cdots}="47", (80,-25)*+{(18,349,5)\cdots}="48", (80,-35)*+{(18,1,65)\cdots}="49", (80,-45)*+{(578,29,5)\cdots}="410", (80,-55)*+{(2,29,169)\cdots}="411", \ar@{-}^{\mu_1}"0";"1"\ar@{-}^{\mu_2}"0";"1'"\ar@{-}_{\mu_3}"0";"1''"\ar@{-}^{\mu_2}"1";"20"\ar@{-}_{\mu_3}"1";"21"\ar@{-}^{\mu_1}"1'";"22"\ar@{-}_{\mu_3}"1'";"23"\ar@{-}^{\mu_1}"1''";"24"\ar@{-}_{\mu_2}"1''";"25"\ar@{-}^{\mu_1}"20";"40"\ar@{-}_{\mu_3}"20";"41"\ar@{-}^{\mu_1}"21";"42"\ar@{-}_{\mu_2}"21";"43"\ar@{-}^{\mu_2}"22";"44"\ar@{-}_{\mu_3}"22";"45"\ar@{-}^{\mu_1}"23";"46"\ar@{-}_{\mu_2}"23";"47"\ar@{-}^{\mu_2}"24";"48"\ar@{-}_{\mu_3}"24";"49"\ar@{-}^{\mu_1}"25";"410"\ar@{-}_{\mu_3}"25";"411"
\end{xy}}\notag
\end{align*}

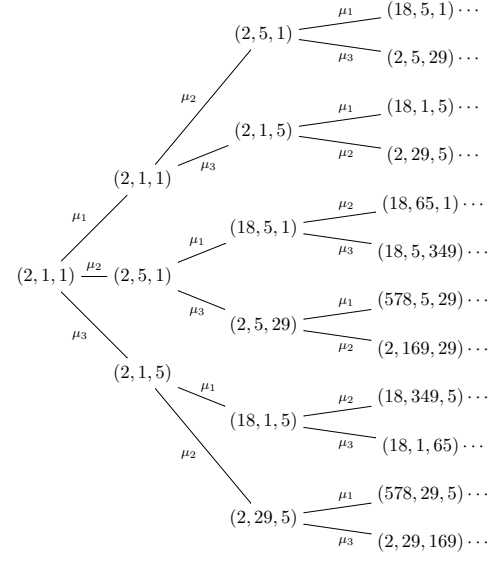
\captionof{figure}{The orbit $\widehat{\Gamma}[(2,1,1)]$ for $k=4$}
\label{1 orbit for $k=4$}
\end{minipage}
\par\vspace{16pt}

\noindent\begin{minipage}[t]{0.49\textwidth}
\centering
\begin{align*}
\scalebox{0.64}{
\begin{xy}(0,0)*+{(1,1,1)}="0",(20,20)*+{(4,1,1)}="1",(20,0)*+{(1,2,1)}="1'",(20,-20)*+{(1,1,2)}="1''",(45,50)*+{(4,17,1)}="20",(45,30)*+{(4,1,17)}="21",(45,10)*+{(9,2,1)}="22",(45,-10)*+{(1,2,5)}="23",(45,-30)*+{(9,1,2)}="24",(45,-50)*+{(1,5,2)}="25",(80,55)*+{(81,17,1)\cdots}="40",(80,45)*+{(4,17,305)\cdots}="41", (80,35)*+{(81,1,17)\cdots}="42", (80,25)*+{(4,305,17)\cdots}="43", (80,15)*+{(9,41,1)\cdots}="44", (80,5)*+{(9,2,85)\cdots}="45", (80,-5)*+{(49,2,5)\cdots}="46", (80,-15)*+{(1,13,5)\cdots}="47", (80,-25)*+{(9,85,2)\cdots}="48", (80,-35)*+{(9,1,41)\cdots}="49", (80,-45)*+{(49,5,2)\cdots}="410", (80,-55)*+{(1,5,13)\cdots}="411", \ar@{-}^{\mu_1}"0";"1"\ar@{-}^{\mu_2}"0";"1'"\ar@{-}_{\mu_3}"0";"1''"\ar@{-}^{\mu_2}"1";"20"\ar@{-}_{\mu_3}"1";"21"\ar@{-}^{\mu_1}"1'";"22"\ar@{-}_{\mu_3}"1'";"23"\ar@{-}^{\mu_1}"1''";"24"\ar@{-}_{\mu_2}"1''";"25"\ar@{-}^{\mu_1}"20";"40"\ar@{-}_{\mu_3}"20";"41"\ar@{-}^{\mu_1}"21";"42"\ar@{-}_{\mu_2}"21";"43"\ar@{-}^{\mu_2}"22";"44"\ar@{-}_{\mu_3}"22";"45"\ar@{-}^{\mu_1}"23";"46"\ar@{-}_{\mu_2}"23";"47"\ar@{-}^{\mu_2}"24";"48"\ar@{-}_{\mu_3}"24";"49"\ar@{-}^{\mu_1}"25";"410"\ar@{-}_{\mu_3}"25";"411"
\end{xy}}\notag
\end{align*}

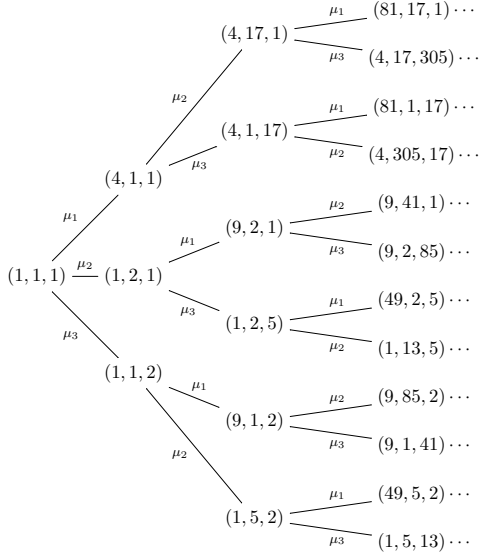
\captionof{figure}{The orbit $\widehat{\Gamma}[(1,1,1)]$ for $k=5$}
\label{1 orbit for $k=5$}
\end{minipage}
\endgroup

\newpage
\begingroup
\setstretch{1.11}

\endgroup
	
\end{document}